\documentclass[11pt]{amsart}

\usepackage{amsmath,amssymb,amscd,amsthm,graphicx}
\usepackage[toc,page]{appendix}

\usepackage[all]{xy}
\usepackage{hyperref}

\numberwithin{equation}{section} 
\theoremstyle{plain}
\newtheorem{theorem}[equation]{Theorem}

\newtheorem{proposition}[equation]{Proposition}
\newtheorem{lemma}[equation]{Lemma}
\newtheorem{corollary}[equation]{Corollary}

\newtheorem{condition}[equation]{Condition}

\theoremstyle{remark}
\newtheorem{remark}[equation]{Remark}

\theoremstyle{definition}
\newtheorem{definition}[equation]{Definition}

\newtheorem*{question*}{Question}

\newcommand{\ra}{\rightarrow}

\newcommand{\C}{{\mathcal C}}

\newcommand{\F}{{\mathcal F}}

\renewcommand{\L}{{\mathcal L}}

\newcommand{\N}{\mathbb N}

\renewcommand{\P}{{\mathcal P}}

\newcommand{\R}{\mathbb R}

\newcommand{\T}{{\mathcal T}}

\newcommand{\Z}{\mathbb Z}
\newcommand{\PLF}{(\P,\L,\F,\pi_\P,\pi_\L)}
\newcommand{\PLFt}{(\P,\L,\F,\pi^t_\P,\pi^t_\L)}

\def\De{\Delta}

\def\ra{\rightarrow}

\usepackage{color}

\def\XXint#1#2#3{{\setbox0=\hbox{$#1{#2#3}{\int}$}
     \vcenter{\hbox{$#2#3$}}\kern-.5\wd0}}

\makeindex
\begin{document}
\title{Curve Shortening Flow and Smooth Projective Planes}
\author{Yu-Wen Hsu}
\address{DEPARTMENT of MATHEMATICS, YALE UNIVERSITY, New Haven, CT 06511}
\email{yu-wen.hsu@yale.edu}
\date{\today}

\addcontentsline{toc}{section}{Abstract}

\begin{abstract}


In this paper, we study a family of curves on $\mathbb{S}^2$ that defines a two-dimensional smooth projective plane. We use curve shortening flow to prove that any two-dimensional smooth projective plane can be smoothly deformed through a family of smooth projective planes into one which is isomorphic to the real projective plane. In addition, as a consequence of our main result,  we show that any two smooth embedded curves on $\mathbb{RP}^2$ which intersect transversally at exactly one point converge to two different geodesics under the flow.

\end{abstract}
\thispagestyle{empty}
\newpage

\maketitle

\thispagestyle{empty}

\thispagestyle{empty}

\section{Introduction}
\subsection{Overview}
The subject of smooth projective planes is intriguing in the field of geometric topology. In contrast to the studies of topological projective planes, see \cite{Sal}, the theory of smooth projective planes is not as well developed. The earliest papers considering differentiable structure on topological projective planes were due to Breitsprecher \cite{Bre67}, \cite{Bre71}, \cite{Bre72} and Betten \cite{Betten} in the late 60's and early 70's. The systematical studies of this subject were first given in the theses of Otte \cite{Otte92} and B{\"o}di \cite{bodi96}. Some characterizations of smooth projective planes were contributed by Linus Kramer, Richard B{\"o}di, Stefan Immervoll et al., see \cite{KL94}, \cite{Richard}, \cite{immer10}.

McKay proved [\cite{Mckay}, Theorem 12] that every regular four-dimensional projective plane can be deformed through a family of regular four-dimensional projective planes into one which is isomorphic to the complex projective plane. McKay applied the Radon transform and defined plane curves in a regular projective plane of dimension four (or more) to show that the exterior differential systems for those curves are elliptic. The proof of this deformation result for dimension four was based on the ellipticity argument and the general theory of pseudocomplex structures which was developed in \cite{Mcduff} and \cite{Mckay}.  However, this cannot be applied to the case of two-dimensional smooth projective plane since the Radon transform and plane curves are not well defined. In this paper, we use curve shortening flow (CSF) to prove the analogous result:
 \begin{theorem}\label{MT} There is a smooth homotopy of two-dimensional smooth projective planes between any two-dimensional smooth projective plane and the real projective plane.\end{theorem}

If $C$ is the space of curves on $\mathbb{RP}^2$ with the standard metric, any two-dimensional smooth projective plane $X$ corresponds to a two-dimensional submanifold $\mathbb{X}$ in $C$ such that any two distinct elements of $\mathbb{X}$ correspond to curves which intersect transversely and exactly once. Denote the submanifold in $C$ corresponding to $\mathbb{RP}^2$ by $\mathbb{G}$, which is comprised of the closed geodesics in $\mathbb{RP}^2$. Therefore, producing the desired one-parameter family of maps from $X$ to $\mathbb{RP}^2$ is equivalent to studying the evolution from $\mathbb{X}$ to $\mathbb{G}$.

CSF shortens any smooth curve by moving it in the direction of its curvature vector field. Gage proved [\cite{Gage}, Theorem 5.1] that any smooth embedded curve on $\mathbb{S}^2$ which bisects the surface area converges to a unique great circle under the flow. Hence it is not hard to believe that one can smoothly flow $\mathbb{X}$ into $\mathbb{G}$ by CSF. Nevertheless, there are several issues arising from flowing curves simultaneously. Firstly, we need to prove that the limit exists in the smooth topology (Gage's result implies only pointwise convergence).  Next, to see a family of curves defines a smooth projective plane at any time, we need to show that CSF preserves the property of transversal intersection between any pair of curves in the family. It is known that this property holds for any $t\geq 0$ as we will discuss in \S \ref{zeros}. This suggests CSF to be a natural tool for this problem. We prove that the transversal intersection is actually preserved in the limit; this turns out to be far more delicate to prove.

Our proof of Theorem~\ref{MT} has the following Corollary and we are not aware of any other proof of this result.
\begin{corollary}
\label{NV}
Any two smooth embedded curves on $\mathbb{RP}^2$ which intersect transversally at a single point converge to two distinct geodesics under curve shortening flow.
\end{corollary}

We prove this corollary by embedding the two curves into a family of curves which defines a smooth projective plane.

\subsection{Organization of the paper}
In section \ref{s2}, we review some background material and show some relevant formulae. In section \ref{s3}, we first prove 
a convergence result for evolving a compact smooth family of curves on $\mathbb{S}^2$ by CSF; note that this convergence result applies to any such a family, it does not require the assumption that the family defines a smooth projective plane. Then, we restrict attention to a family of curves which defines a smooth projective plane and show that one gets a smooth homotopy of smooth projective planes, after reparametrization to the time interval $[0,1)$. In section \ref{s4}, we extend the smooth homotopy of smooth projective planes to the closed time interval $[0,1]$ by analyzing the linearized curve shortening equation and present the proof of our main result  Theorem~\ref{MT}.

\subsection{Acknowledgements} The author would like to thank Bruce Kleiner for introducing her to
the problem, and for his guidance and direction during the entire
project. The author would also like to acknowledge  numerous helpful
conversations with Philip Gressman, Subhojoy Gupta, Joseph Lauer, Yair
Minsky and Rishi Raj.
\section{Preliminaries}
\label{s2}

\subsection{Smooth projective planes}
The classical example of a smooth projective plane of dimension two is the real projective plane, $\mathbb{RP}^2$. It can be thought of as the set of lines through the origin in $\mathbb{R}^3$. A \textit{line} in $\mathbb{RP}^2$ is then the set of lines through the origin in $\mathbb{R}^3$ that lie in the same plane. Any two planes through the origin in $\mathbb{R}^3$ intersect at a unique line through the origin in $\mathbb{R}^3$. Lines through the origin in $\mathbb{R}^3$ can be thought of as \textit{points} in $\mathbb{RP}^2$. Any two \textit{lines} intersect at a unique \textit{point} and any two \textit{points} can be joined by a unique \textit{line}. Alternatively, one can think of $\mathbb{RP}^2$ as the unit sphere $\mathbb{S}^2$ with antipodal points identified. In this setting, a \textit{line} in $\mathbb{RP}^2$ is a great circle and a \textit{point} is a pair of antipodal points on $\mathbb{S}^2$.

 \begin{definition}[Projective planes] A {\bf projective plane} is a triple $(\P,\L,\F)$ which consists of the point space $\P$, the line space $\L$ and the flag space $\F\subset \P\times \L$ such that the following axioms are satisfied.
\begin{enumerate}
\item Any two points $p$, $q$ in $\P$ can be joined by a unique line $L=p\;\vee \;q \in\L$ that is $(p,L)$ and $(q,L)$ are in $\F$.
\item Any two lines $l_1$, $l_2$ in $\L$ intersect at a unique point $p= l_1\wedge l_2 \in \P$ that is $(p,l_1)$ and $(p,l_2)$ are in $\F$.
\item There are 4 points, no three of which are on the same line.
\end{enumerate}
\end{definition} 
 \begin{definition}[Smooth projective planes] \label{SP2}
 A projective plane is called a {\bf smooth projective plane} if $\P$ and $\L$ are smooth manifolds and the maps $\vee:\P\times \P\setminus \bigtriangleup(\P)\to \L$ and $\wedge:\L\times \L\setminus\bigtriangleup(\L) \to \P$ are smooth, where $\De$ denotes the diagonal in $\P\times \P$ .
\end{definition}

\begin{theorem} [Freudenthal \cite{Freudenthal}]The dimension of a smooth projective plane is
either 0, 2, 4, 8 or 16.\end{theorem}

\begin{theorem}[\cite{SH} 51.29]\label{diffreal}Two dimensional smooth projective planes are diffeomorphic to the real projective plane.\end{theorem}
\begin{theorem}[Mckay \cite{Mckay}] Every smooth projective plane of dimension 4 is diffeomorphic to the complex projective plane.\end{theorem}
\begin{remark}Every smooth projective plane of positive dimension is diffeomorphic to its model.
 For dimensions 8 and 16, this was proven by Kramer and Stolz \cite{Kramer}.\end{remark}
 
\begin{definition}[Smooth generalized plane, \cite{Richard}]
Suppose that $(\P,\L,\F)$ is a projective plane and $\P$ and $\L$ are $2n$-dimensional closed smooth manifolds and $\F\subset \P\times \L$ is a $3n$-dimensional closed smooth embedded submanifold so that the canonical projections $\pi_p: \F\to \P:(p,l)\mapsto p$ and $\pi_l: \F\to \L:(p,l)\mapsto l$ are both submersions. Then it is a {\bf smooth generalized plane}. 
\end{definition}

\begin{definition}[Point rows and line pencils]
Let $(\P,\L,\F)$ be a smooth generalized plane. For $\ell\in \L$, we call the set $P_\ell=\pi_p(\pi^{-1}_\L(\ell))$ the {\bf point row} associated with $l$. For $p\in \P$, we call the set $L_p=\pi_\L(\pi^{-1}_p(p))$ the {\bf line pencil} through $p$. 
\end{definition}

\begin{theorem}[\cite{KL94}, p86]
Let $(\P,\L,\F)$ be a smooth projective plane of dimension $2n$. Then the point rows and the line pencils are smoothly embedded $n$-sphere. The spaces $\P$, $\L$, and $\F\subset \P\times\L$ are compact connected smooth manifolds of dimension 2n, 2n, 3n, respectively. Moreover, $\pi_\P:\F\to\P$ is  a locally trivial smooth $n$-sphere bundle.  
\end{theorem} 

\begin{definition}[\cite{Richard}]
Two lines $\ell_1$ and $\ell_2$ of a smooth generalized plane are said to {\bf intersect transversally} in some point $p$ if the associated point rows $P_{\ell_1}$ and $P_{\ell_2}$ intersect transversally in $p$ as submanifolds in $\P$, i.e. their tangent spaces in $\P$ span the tangent space $T_p\P$. Dually, we say that two line pencils intersect transversally in $\ell$ if their tangent spaces in $\L$ span the tangent space $T_\ell\L$.
\end{definition}

In \cite{Richard} Corollary 1.4, B{\"o}di and Immervoll  have proved that a smooth projective plane is a smooth generalized plane for which any two lines are transverse and the pencils of any two points are transverse and \textit{vice versa}. The condition on the transversality of any two point rows implies that the intersection map is locally well defined and smooth (\cite{Richard} Theorem 1.2).  One can therefore use this characterization as a definition of smooth projective planes.

\begin{definition}[Smooth projective planes, alternate definition]
 \label{SP1}
Let $(\P,\L,\F)$ be a smooth generalized plane. Suppose that $\pi_{\P}:\F\to\P$ and $\pi_{\L}:\F\to\L$ are submersions between compact manifolds so that:
 \begin{itemize}
\item[SPP1:] For any two distinct lines $l_1$, $l_2\in\L$, $\pi_{\P}(\pi_{\L}^{-1}(l_1))$ intersects $\pi_{\P}(\pi_{\L}^{-1}(l_2))$ transversely in a single point $p_0\in\P$.  
\item[SPP2:] For any two distinct points $p_1$, $p_2\in\P$, $\pi_{\L}(\pi_{\P}^{-1}(p_1))$ intersects $\pi_{\L}(\pi_{\P}^{-1}(p_2))$ transversely in a single point $l_0\in\L$.
\end{itemize}
Then we call the tuple $\PLF$ a {\bf smooth projective plane}.

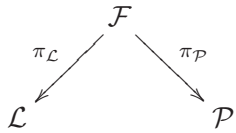
\begin{displaymath}
    \xymatrix{
        & \ar[dl]_{\pi_\L} \F \ar[dr]^{\pi_{\P}}\\
                           \L&& \P}
\end{displaymath}
\end{definition}
Recall that for any smooth projective plane, $\P$ and $\L$ are both diffeomorphic to $\mathbb{RP}^2$. The point rows and line pencils are smooth embedded 1-sphere. 
\begin{remark}
The real projective plane is a 2-dimensional smooth projective plane whose point rows (lines) are geodesics of the standard metric of $\mathbb{RP}^2$. 
\end{remark}

\begin{remark}
\label{nullhomo}
Point rows of smooth projective plane are not null-homotopic in $\mathbb{RP}^2$. Suppose there exists a line $\ell_1$ such that $P_{\ell_1}=\pi_{\P}(\pi_{\L}^{-1}(\ell_1))$ is null-homotopic in $\mathbb{RP}^2$. Then $P_{\ell_1}$ divides $\mathbb{RP}^2$ into two connected components $D_1$ and $D_2$. Choose two points $p$, $q$ in $\P$ so that $p$ is in $D_1$ and $q$ is in $D_2$. The projective structure implies that there must exist a unique line $\ell_2$ join $p$ and $q$. Moreover, since point rows intersect transversely, $P_{\ell_2}$ must intersect $P_{\ell_1}$ at more than one point. This contradicts with SPP1 in Definition~\ref{SP1}.
\end{remark}

Our goal is to smoothly deform any smooth projective plane to the standard one. It is then necessary to deform the set of point rows into the set of geodesics. To be more precise, we give the definition of such a deformation as below.
\begin{definition}[Smooth homotopy of smooth projective planes]
\label{isotopy}
A {\bf smooth homotopy of smooth projective planes} consists of a smooth projective plane $\PLF$, and smooth homotopies $\pi_{\L}^t$, $\pi_{\P}^t$, for $t\in [0,1]$, such that for every $t\in [0,1]$, the tuple $\PLFt$ is a smooth projective plane. 
\end{definition}

\subsection{Curve shortening flow}
Curve shortening flow is a heat-type geometric flow which evolves curves in the direction of their curvature vector field. In this paper, we only consider curves on round $\mathbb{S}^2$ or $\mathbb{RP}^2$, so we define CSF and formulate Grayson's result of the long time existence solution to CSF on a compact surface $M^2$ as below:

\begin{definition}
Let $\gamma_0: \mathbb{S}^1 \to M^2$ be a smooth curve. We say that $\gamma: \mathbb{S}^1\times [0,t_0) \to M^2 $ is the solution to the {\bf curve shortening equation} with the initial data $\gamma_0$ if it satisfies 
\begin{eqnarray}
\label{CSF}
\frac{\partial \gamma}{\partial t} = \kappa N,
\end{eqnarray}
for all $t\in[0,t_0)$, where $[0,t_0)$ is the maximal interval on which the solution can be defined, and $\kappa$ is the geodesic curvature of $\gamma$ and $N$ is its unit normal vector.  
\end{definition}
\begin{theorem}[Theorem 0.1, \cite{Grayson89}]
\label{gray}
Let $\gamma_0:\mathbb{S}^1\to M^2$ be a smooth curve, embedded in $M^2$. Then the solution to (\ref{CSF}), $\gamma(t):\mathbb{S}^1\to M^2$, exists for $t\in[0,t_0)$. If $t_0$ is finite, then $\gamma(t)$ converges to a point. If $t_0$ is infinite, then the curvature of $\gamma(t)$ converges to zero in the $C^\infty$ norm.
\end{theorem}
In the round sphere case, $M^2=\mathbb{S}^2$, the fact that the geodesic curvature tends uniformly to zero implies that the curve $\gamma(t)$ is, for large $t$, close to some great circle $\bar\gamma(t)$; however, Grayson's methods do not imply that $\gamma(t)$ converges to a (fixed) great circle, i.e. that one may choose $\bar\gamma(t)$ to be independent of $t$. This was proven by Gage \cite{Gage}:
\begin{theorem}[Theorem 5.1, \cite{Gage}]
\label{unique}
A simple closed curve $\gamma$ on the sphere of radius $1/C$ which divides the sphere into two pieces of equal area and whose total space curvature $\int (\kappa^2+C^2)^{1/2}ds$ is less than $3\pi$ converges to a great circle under curve shortening flow.
\end{theorem}

We now recall some computations on the evolution equations for curvature functions from \cite{Gage} and \cite{Grayson89}. Let $\gamma(x,t):S^1\times [0,\infty)\to S^2$ be the solution to the CSF with a smooth initial condition. The arc length $s$ is defined by $ds=|\frac{\partial \gamma}{\partial x} |dx$. Let $\nu=\displaystyle|\partial \gamma/\partial x|$. A computation shows $\nu_t=-\kappa^2\nu$. The variables $x$ and $t$ are independent so $\displaystyle\partial x\partial t=\partial t\partial x$. For the arc-length parameter $s$, if we switch the order of differentiating $s$ and $t$, the following equation has to be satisfied:
\begin{equation}
\label{switchst}
\frac{\partial}{\partial t}\frac{\partial}{\partial s}=\frac{\partial}{\partial s}\frac{\partial}{\partial t } + \kappa^2\frac{\partial}{\partial s}.
\end{equation}
The curvature $\kappa$ evolves according to
\begin{eqnarray}
\label{knest0}
\frac{\partial}{\partial t} k =k^{(2)}+k+k^3.
\end{eqnarray}
For $n\geq 1$, we let $\kappa^{(n)}$ stand for $\displaystyle\frac{\partial^n \kappa}{\partial s^n}$ (note that $\kappa^{(0)}=\kappa$). By \eqref{switchst} and \eqref{knest0}, one can derive that $\kappa^{(n)}$ evolves according to
\begin{equation}
\begin{split}
\label{knest}
\frac{\partial}{\partial t}\kappa^{(n)}=\kappa^{(n+2)}+\kappa^{(n)}+(3+n)(\kappa^2)(\kappa^{(n)})+\sum_{\substack{{i+j+r=n-1} \\0\leq i,j,r\leq n-1}}C_{ijr} \kappa^{(i)}\kappa^{(j)}\kappa^{(r)},
\end{split}
\end{equation}
where $C_{ijr}$'s are constants depending on $n$. Integration by parts yields:
\begin{equation}
\begin{split}
\label{intk1}\frac{\partial}{\partial t} \int (\kappa^{(n)})^2ds =&\int \Big(-2(\kappa^{(n+1)})^2+2(\kappa^{(n)})^2+(2(3+n)-1)(\kappa^2)(\kappa^{(n)})^2\\
&+2\kappa^{(n)}\sum_{\substack{{i+j+r=n-1} \\0\leq i,j,r\leq n-1}}C_{ijr} \kappa^{(i)}k^{(j)}\kappa^{(r)}\Big)ds. 
\end{split}
\end{equation}

Gage \cite{Gage} observed the following,
\begin{lemma}
\label{intkiszero}
Any simple closed smooth area-bisecting curve on $\mathbb{S}^2$ remains area-bisecting under curve shortening flow for all time.
\end{lemma}
\begin{proof}
By Gauss-Bonnet Formula, any curve on $\mathbb{S}^2$ which bisects the surface area, its curvature satisfies $\int \kappa ds =0$ and vice versa. Flowing such a curve by CSF, we have\begin{eqnarray}
\begin{split}
\frac{d}{dt}\int \kappa ds&=\int \frac{\partial \kappa}{\partial t}ds+\int \kappa \frac{\partial}{\partial t} ds\\
&=\int (\kappa^{(2)}+\kappa+\kappa^3) ds+\int \kappa(-\kappa^2) ds\\
&=\int \kappa^{(2)}+\kappa ds.
\end{split}
\end{eqnarray}
Because $\kappa^{(n)}$ is $2\pi$ periodic, using integration by parts we get $ \int \kappa^{(2)} ds=0$. The solution to the following initial value problem
\begin{eqnarray}
\begin{split}
\frac{d}{dt}\int_{\gamma_t} \kappa ds&=\int_{\gamma_t} \kappa ds\\
\int_{\gamma_0}\kappa ds&=0
\end{split}
\end{eqnarray}
equals $\int_{\gamma_t}\kappa ds=e^t\int_{\gamma_0}\kappa ds=0$.
Hence the condition $\int_{\gamma_t} \kappa=0$ is preserved for any time.
\end{proof}
We recall Wirtinger's inequality (Poincar\'e inequality of dimension one) and Gronwall's inequality here:
\begin{lemma}[Wirtinger's inequality, \cite{Hardy}]
\label{Pinequalityonedim}
Let $f$ be a periodic real function with period $2\pi$ and let $f'\in L^2$. Then, if $\int_0^{2\pi} f (x) dx=0$, the following inequality holds
\begin{equation}
\label{Wirting2}
\int_0^{2\pi} f(x)^2 dx\leq \int_0^{2\pi} (f'(x))^2 dx.
\end{equation}  
The equality holds if and only if $f=A\sin x+B\cos x$, where $A$ and $B$ are constants. 
\end{lemma}
\begin{remark}
If the function $f$ in Lemma~\ref{Pinequalityonedim} is smooth, then for every $n$,
\begin{equation}
\label{Wirting}
\int_0^{2\pi} (f^{(n-1)}(x))^2 dx\leq \int_0^{2\pi} (f^{(n)}(x))^2 dx.
\end{equation} 
 \end{remark}

\begin{lemma}[Gronwall's inequality]
\label{Gronwall}
Let $\eta(\cdot)$ be a nonnegative, absolutely continuos function on $[0,T]$ which satisfies for a.e. $t$ the differential inequality
\begin{equation}
\eta'(t)\leq \phi(t)\eta(t)+\psi(t).
\end{equation}
where $\phi(t)$ and $\psi(t)$ are nonnegative, integrable functions on $[0,T]$. Then
\begin{equation}
\eta(t)\leq e^{\int_0^t\phi(s) ds}\Big(\eta(0)+\int_0^t\psi(s)ds\Big).
\end{equation}
\end{lemma}

\subsection{Curve shortening equation in local coordinates}
In \cite{Angenent}, the corresponding PDE for CSF in local coordinates was derived in a general form ((67) in page 39). Here we consider a parametrization of $\mathbb{S}^2$ at $p$ as follows:
\begin{eqnarray}
\label{graph}
\mathbf{x}(x,z)&=&\Big(
\frac{\cos x}{\sqrt{1+z^2}}, 
\frac{\sin x}{\sqrt{1+z^2}},\frac{z}{\sqrt{1+z^2}}\Big),
\end{eqnarray}
where $x\in [0,2\pi]$, $z\in [-r,r]$, for some small $r>0$, and the metric $g$ on $\mathbb{S}^2$ can be written as
\[
g=\left( \begin{array}{cc}
       \frac{1}{1+z^2} & 0        \\
     0 & \frac{1}{(1+z^2)^2}
     \end{array}\right)
     \]
for \begin{eqnarray*}
E&=&\frac{\partial}{\partial x}\mathbf{x}(x,z)\cdot\frac{\partial}{\partial x}\mathbf{x}(x,z) =\frac{1}{1+z^2}\\
F&=&0\\
G&=&\frac{\partial}{\partial z}\mathbf{x}(x,z)\cdot\frac{\partial}{\partial z}\mathbf{x}(x,z) =\frac{1}{(1+z^2)^2}.
\end{eqnarray*}

Let $\sigma:\mathbb{S}^1\times [-r,r]\to \mathbb{S}^2$ be a local diffeomorphism such that the parametrization of $\gamma_0$ is given by $x\mapsto \sigma(x,0)$. Firstly, we compute the unit tangent $T$ to the graph $\{(x,h(x))|x\in \mathbb{S}^1\}$, and the geodesic curvature $\kappa$ of it as below:
\begin{equation}
\label{unitvector}
T=\frac{1}{\nu}(\partial_x+h_x\partial_z),\;\;\nu=\sqrt{\frac{1+h^2+(h_x)^2}{(1+h^2)^2}}.
\end{equation}
Since the Christoffel symbols of the Riemannian connection are $\Gamma_{xx}^z=h$, $\Gamma_{hh}^h=\frac{-2h}{1+h^2}$, $\Gamma_{xz}^x=\frac{-h}{(1+h^2)}$, we have
\begin{eqnarray*}
\begin{split}
\nabla_TT=&\frac{1}{\nu^2}[(T^x\partial_zT^x+T^xT^z\Gamma_{xz}^x)\partial_x+T^z\partial_zT^z+T^xT^z\Gamma_{zx}^z)\partial_x\\
&+T^z\partial_zT^z+T^zT^x\Gamma_{xz}^z+T^xT^z\Gamma_{zz}^z)\partial_z+T^x\partial_xT^z+T^xT^x\Gamma_{xx}^z)\partial_z]\\
=&\frac{1}{\nu^2}[\frac{-2hh_x}{(1+h^2)\partial_x}+(h_{xx}+h-\frac{2hh_x^2}{(1+h^2)})\partial_z].
\end{split}
\end{eqnarray*}
Hence $\kappa=T\wedge\nabla_TT$ is
\begin{equation}
\label{kh}
\frac{1}{\nu^3}\left( \begin{array}{cc}
      1& h_x        \\
     \frac{-2h_xh}{1+h^2}& h_{xx}+h-\frac{2h_x^2h}{1+h^2}
     \end{array}\right)=\frac{1}{\nu^3}(h_{xx}+h)
.\end{equation}

Flowing $\gamma$ by CSF, the corresponding time evolution equation for $h$ can be derived by $
T\wedge h_t\partial_z=T\wedge\nabla_TT$. Since $$T\wedge h_t\partial_z=
\frac{1}{\nu^3}\left( \begin{array}{cc}
      1& h_x        \\
 0& h_t
     \end{array}\right)=
\frac{1}{\nu}h_t$$ together with (\ref{kh}) we have
\begin{equation} \label{h}
h_t=\frac{(1+h^2)^2}{1+h^2+(h_x)^2}(h_{xx}+h).
\end{equation}
For instance, it is clear that under CSF, great circles on $\mathbb{S}^2$ will not move at all (geodesic curvature is 0). They correspond to solutions $h(x)=a \sin x+b\cos x$, where $a$, $b$ are constants, to (\ref{h}). We can look at two easy examples: if $h(x)=0$ ($a,\;b$ are both 0), the corresponding great circle is the equator and if $h(x)=\cos x$, we obtain the corresponding great circle by intersecting the plane $x=z$ with the sphere. 

\subsection{Zeros of linear parabolic PDEs and intersection points of pairs of solutions to CSF}
\label{zeros}
In \cite{Angenent88}, the author studies the zero set of a solution $u(x,t)$ of the equation
\begin{equation}
\label{parabolic}
u_t=\mathbf{a}(x,t) u_{xx}+\mathbf{b}(x,t)u_x+\mathbf{c}(x,t) u
\end{equation}
under the assumptions
\begin{enumerate}
\item[a1:] $\mathbf{a}$, $\mathbf{a}^{-1}$, $\mathbf{a}_{t}$, $\mathbf{a}_{x}$, and $\mathbf{a}_{xx}$ $\in L^{\infty}$
\item[a2:] $\mathbf{b}$, $\mathbf{b}_t$, and $\mathbf{b}_x$ $\in L^{\infty}$
\item[a3:] $\mathbf{c}\in L^{\infty}$ 
\end{enumerate}
where the number of zeros of $u(\cdot,t)$ is defined to be the supremum over all $k$ such that there exists $0<x_1<x_2<...<x_k<1$ with $u(x_i,t)\cdot u(x_{i+1},t)<0$ for $i=1,2,...,k-1$. Let $z(t)$ denote this supremum. The following theorem says that $z(t)$ is a nonincreasing function with time.
\begin{theorem}[Theorem C, \cite{Angenent88}]
\label{interpara}
Let $u:[0,1]\times [0,T]\to\R$ be a bounded solution of (\ref{parabolic}) which satisfies either Dirichlet, Neumann or periodic boundary condition. Assume that the coefficients $a$, $b$ and $c$ satisfy assumptions a1, a2 and a3, and in addition, in the case of Neumann boundary conditions, assume that $\mathbf{a}\equiv 1$, $\mathbf{b}\equiv 0$. Let $z(t)$ denote the number of zeros of $u(\cdot, t)$ in $[0,1]$, then
\begin{itemize}
\item[(a)] for $t>0$, $z(t)$ is finite
\item[(b)] if $(x_0,t_0)$ is a multiple zero of $u$ (i.e. if $u$ and $u_x $ vanish simultaneously), then for all $t_1<t_0<t_2$, $z(t_1)>z(t_2)$. 
\end{itemize}
\end{theorem}
\begin{remark}
If a new point of intersection is developed, it must be a multiple zero. This contradicts with (b). Therefore $z(t)$ cannot increase with time. 
\end{remark}

In \S\ref{LCSFequation}, we will derive the linearized curve shortening equation (LCSF):
\begin{equation}
\label{linearized}
v_t=(1+a(x,t)) v_{xx}+b(x,t)v_x+(1+c(x,t)) v,
\end{equation}
where $||a||_{C^k}$, $||b||_{C^k}$, and $||c||_{C^k}$ are all less than $\epsilon_k e^{-t}$ for some $\epsilon_k=\epsilon_k(k)$ sufficiently small. If we set $\mathbf{a}(x,t)=1+a(x,t)$, $\mathbf{b}(x,t)=b(x,t)$ and $\mathbf{c}(x,t)=1+c(x,t)$, then conditions a1, a2, and a3 are satisfied. Moreover, since we only consider smooth embedded area-bisecting curves on $\mathbb{S}^2$, the boundary condition for $v$ is periodic. We can apply Theorem~\ref{interpara} to the LCSF and conclude the following:
\begin{proposition}
\label{transversezeros}
The number of transverse zeros of the solution $v$ to the linearized flow (\ref{linearized}) cannot increase with time. 
\end{proposition}

Let $\alpha_0$ and $\beta_0$ be two smooth embedded curves on $\mathbb{RP}^2$ which intersect transversely at exactly one point. Let $\alpha_t$ and $\beta_t$ be the solutions to CSF with initial conditions $\alpha_0$ and $\beta_0$. By Theorem 1.3 in \cite{Angenent91}, the number of transverse intersections cannot increase with time. On the other hand, it won't decrease to zero since any curve lifted from $\mathbb{RP}^2$ to $\mathbb{S}^2$ bisects the surface area, and by Lemma~\ref{intkiszero} it remains to do so under CSF. Therefore one can conclude: 
\begin{proposition}
\label{interc}
Suppose $\alpha_t$ and $\beta_t$ are defined as above. At any $T>0$, the number of transverse intersection of $\alpha_T$ and $\beta_T$ stays one.  
\end{proposition}
\begin{remark} Two different curves cannot coincide at any finite time since they would have infinitely many oriented tangencies which contradicts Theorem 1.1 in \cite{Angenent91}, page 175. 
\end{remark}

\section{Constructing a smooth homotopy using CSF}
\label{s3}

 \subsection{Long time behavior of an individual curve under CSF}
 In this section, $\gamma$ is assumed to be a smooth embedded closed curve in $\mathbb{S}^2$ that divides the surface area into two equal pieces and $\gamma_t$ denotes the solution to the CSF with initial data $\gamma_0=\gamma$.
  \label{ltb}
\begin{theorem}
\label{EX}
For any integer $k$, there is an $\epsilon_k>0$ such that for any $0<\epsilon\leq\epsilon_k$, there is an $\delta=\delta(\epsilon)>0$ so that if $\gamma$ is $C^{k+2}$ $\delta$-close to a great circle, then there is a (perhaps different) great circle $\gamma_g$ such that $d_{C^k}(\gamma_t, \gamma_g)\leq  2\epsilon e^{-t} $ for all $t\geq 0$.
\end{theorem}
\begin{remark}
In this case, we say that $\gamma_t$ converges uniformly exponentially to $\gamma_g$ in the $C^k$ norm.
\end{remark}
\begin{remark}
In fact, one can prove $d_{C^k}(\gamma_t, \gamma_g)\leq  2\epsilon e^{-\eta t} $ for some other constant $1<\eta<3$ using the same argument in our proof. 
\end{remark}

Let $\gamma_g(\theta): \mathbb{S}^1\to\mathbb{S}^2$ be a parametrization of a great circle, where $\theta$ is the usual angle parameter on $\mathbb{S}^1\subset\mathbb{R}^2$. For each great circle, there is a $C^2$ local diffeomorphism $\sigma_g:\mathbb{S}^1\times (-r,r)\to\mathbb{S}^2$ such that if $\gamma$ is $C^2$ close to $\gamma_g$, then there exists a $C^2$ function $h:[0,2\pi]\to(-r,r)$ so that $\gamma(\theta)=\sigma_g(\theta,h(\theta))$. In particular $\gamma_g(\theta)=\sigma_g(\theta,0)$. Note that we will show in Proposition~\ref{he0} that for any $\gamma_0$ which is $C^{k+2}$ close to a great circle, it stays close to a great circle for all time under the flow.

For a curve which is $C^2$ close to a great circle, there are two natural parametrizations, one by $\theta$, and one by arclength, $s=s(\theta)$ with $ds=|\gamma'(\theta)|d\theta$. The following Lemma says that the closeness does not depend on which of those two parametrizations we choose.

\begin{lemma}
\label{dd'}
For any $0<\delta<\frac{1}{2\pi}$, if $\gamma$ is $C^k$ $\delta$-close to a geodesic $\gamma_g$ when parametrized by $\theta$, then $\gamma$ is $C^k$ $2\delta$-close to $\gamma_g$ when parametrized by its arclength $s$, and vice versa.
\end{lemma}
\begin{proof}
By (\ref{unitvector}) and a straightforward computation, we have that for any $0<\delta<\frac{1}{2\pi}$, if $||h(\theta)||_{C^k}\leq\delta$ then $||s(\theta)-\theta||_{C^k}\leq \delta$. Therefore, the Lemma follows.
\end{proof}
The following inequality plays a key role in our proof of Theorem~\ref{EX}.
\begin{lemma}[Inequality of Poincar\'e Type for $\kappa$]
\label{Poin1}
There exists a constant $\delta_0$ such that for any $0<\delta\leq \delta_0$ if $\gamma$ is $C^1$ $\delta$-close to a great circle, then 
\begin{eqnarray}
\label{Poin}
\int _{\gamma}(\kappa) ^2 ds \leq \frac{2}{5}\int_{\gamma} (\kappa^{(1)}) ^2 ds.
\end{eqnarray}
\end{lemma} 

\begin{remark}
The constant $\frac{2}{5}$ appearing in \eqref{Poin} is enough for the purpose of the present paper.
\end{remark}

\begin{remark}Inequality (\ref{Poin}) implies the following:
\begin{eqnarray} 
\label{knC}
\int _{\gamma}(\kappa^{(n-1)}) ^2 ds \leq \frac{2}{5}\int_{\gamma} (\kappa^{(n)}) ^2 ds,\;\;\text{for all $n\geq 1$}.
\end{eqnarray}
This can be proved by induction and the inequality:
\begin{eqnarray*}
\label{PCE} 
(\int _{\gamma}(\kappa^{(n-1)}) ^2 ds)^2\leq \int_{\gamma} (\kappa^{(n-2)}) ^2 ds \;\int _{\gamma}(\kappa^{(n)}) ^2 ds, \;\;\text{for $n\geq 2$.}
\end{eqnarray*}
 \end{remark}

In order to prove this lemma, we study the Fourier expansion of $\kappa(s)$ of $\gamma$ in $s$. Let $L$ be the length of $\gamma$, and denote the space of functions $f$ in $C^\infty[0,L]$ satisfying $f(0)=f(L)$ by $\C$. 
Let $S_0$ be the subspace of $\C$ spanned by the constant functions, $S$ be the subspace of $\C$ spanned by $\sin s$ and  $\cos s$, and $S^{\perp}$ be the orthogonal complement of $S_0$ and $S$ in $\C$. Since $\gamma$ divides the surface area into two equal pieces, Gauss-Bonnet formula implies that $\int_\gamma \kappa\; d s=0$, i.e. the projection of $\kappa$ onto $S_0$ is always zero. Hence $\kappa=\kappa_S\oplus \kappa_{S^{\perp}}$. If, in addition, $\kappa$ is orthogonal to $S$, then
\begin{eqnarray}
\label{PoinS}
\int_\gamma({\kappa}^{(n-1)}) ^2 ds \leq\frac{1}{4}\int _\gamma({\kappa}^{(n)}) ^2 ds.
\end{eqnarray}

The following Lemma shows that $\kappa$ is almost orthogonal to $S$.
\begin{lemma}\label{M} Let $\gamma$ be $C^1$ $\delta$-close to a great circle. There are real valued functions $u(s)$ and $w(s)$ so that $\int_\gamma  \kappa  u \;ds =0$, and $\int_\gamma \kappa  w \;ds =0$, where
\begin{eqnarray}
\label{usin}
|| u-\sin s||_{C^0} \leq 2\delta,\; \;|| w-\cos s||_{C^0}
\leq 2\delta.
\end{eqnarray}
\end{lemma}

\begin{proof}
Let $\Gamma_{\phi}$ be a proper variation of $\gamma$ obtained by acting the matrix of rotation about $x$-axis $R_x(\phi)$ on $\gamma$ with $\Gamma_{0}=\gamma$. Then $V_{\gamma}=\frac{d}{d\phi}\arrowvert_{\phi=0}\Gamma_{\phi}$ represents the variation vector field along $\gamma$. Since $R_x$ is an isometry, by the first variation
formula of length,
\begin{eqnarray}
\label{1var}
\begin{split}
0&=\frac{d}{d\phi}\Big|_{\phi=0}L(\Gamma_{\phi})=-\int_{\gamma}\kappa<V_{\gamma},N_{\gamma}>ds
\end{split}
\end{eqnarray}
where $N_{\gamma}$ is the unit normal of $\gamma$.

Define $u=<V_{\gamma},N_{\gamma}>$, then $\int_\gamma \kappa u \;ds=0$. Since $\gamma$ is $C^1$ $\delta$-close to a great circle, we have 
\begin{eqnarray}
\label{2var}
\begin{split}||<V_{\gamma(s(\theta))}, N_{\gamma(s(\theta))}>-<V_{\gamma_g(\theta)},N_{\gamma_g(\theta)}>||_{C^0}\leq\delta,
\end{split}\end{eqnarray}
where $V_{\gamma_g}$ is the variation vector field along $\gamma_g$ generated by $R_x$. Note that we have switched from parametrization with arclength in (\ref{1var}) to parametrization with $\theta$ in (\ref{2var}). We can do this because of Lemma~\ref{dd'}.

One can compute that $<V_{\gamma_g(\theta)},N_{\gamma_g(\theta)}>=\sin \theta$. Also, ${||s(\theta)-\theta||_{C^0}\leq\delta}$ implies $||\sin s(\theta)-\sin \theta||_{C^0}\leq\delta$. Therefore,
\begin{eqnarray*}\begin{split}
&||<V_{\gamma(s(\theta))}, N_{\gamma(s(\theta))}> -\sin s(\theta)||_{C^0}\\
=&||<V_{\gamma(s(\theta))}, N_{\gamma(s(\theta))}>-<V_{\gamma_g(\theta)},N_{\gamma_g(\theta)}>+\sin\theta -\sin s(\theta)||_{C^0}\\
\leq&2\delta.
 \end{split}
 \end{eqnarray*}
This shows that $||u-\sin s||_{C^0}\leq2\delta.$ 

Analogously, $w$ is defined by considering the variation vector field generated by the matrix of rotation about $y$-axis.
\end{proof}

\begin{proof}[Proof of Lemma~\ref{Poin1}]
Let $\kappa(s)=\sum_n \kappa_n e^{ins}$ where $\kappa_n=\frac{1}{L(\gamma)}\int_\gamma\kappa e^{-ins} ds$ is the $n^{th}$ Fourier coefficient (note that $\kappa_0=0$). Then $\int_\gamma \kappa^2 ds=L(\gamma)\sum_n|\kappa_n|^2$. Let $u(s)$ and $w(s)$ be chosen in Lemma~\ref{M}.
\begin{eqnarray}
\begin{split}
|\kappa_1|&=|\frac{1}{L(\gamma)}\int_\gamma\kappa e^{-is} ds|\\
&\leq \frac{1}{L(\gamma)}\int_\gamma | \kappa | |(e^{-is} -(w+iu))|ds\\
&\leq 4\delta (\int_\gamma \kappa^2 ds)^{1/2}.
\end{split}
\end{eqnarray}
Hence $|\kappa_1|^2\leq 16 \delta^2 L(\gamma)\sum_n |\kappa_n|^2$. Choose $\delta_0=\frac{\sqrt{3}}{16\sqrt{L(\gamma)}}$, then for any $0<\delta\leq\delta_0$,
$$
\frac{3}{5}\sum_{|n|\neq 1}|\kappa_n|^2 -2|\kappa_1|^2=\frac{3}{5}\sum_{n}|\kappa_n|^2 -\frac{16}{5}|\kappa_1|^2 \geq 0.
$$
Therefore,
\begin{eqnarray}
\label{PoinSper}
\int_\gamma{\kappa}_S^2 ds \leq \frac{3}{5}\int_\gamma \kappa_{S^\perp}^2 ds.
\end{eqnarray}

By (\ref{PoinS}) and (\ref{PoinSper}), 
\begin{eqnarray*}
\begin{split}
&\int_\gamma \kappa^2ds-\frac{2}{5}\int_\gamma (\kappa^{(1)})^2ds\\
=&\int_\gamma \Big( (\kappa_S)^2+ (\kappa_{S^{\perp}})^2-\frac{2}{5}((\kappa^{(1)})_S))^2
-\frac{2}{5}((\kappa^{(1)})_{S^{\perp}})^2\Big)ds\\
\leq& \;\frac{8}{5}\int_\gamma  (\kappa_{S^{\perp}})^2 ds-\frac{8}{5}\int_\gamma (\kappa_{S^{\perp}})^2ds=0.
\end{split}
\end{eqnarray*}
\end{proof}

For any curve $\gamma_0$ that is $C^1$ close to a great circle, we can assume $\gamma_t$ the solution to CSF satisfying $2\pi\leq L(\gamma_t)\leq 3\pi$ for all $t\geq 0$. 
\begin{lemma}
\label{1}
For any $\delta>0$, there is an $\epsilon=\epsilon(\delta)>0$ such that if $\int_{\gamma} (\kappa^{(k)})^2 ds< \epsilon^2$, then $\gamma$ is $C^{k+1}$ $\delta$-close to a great circle. 
\end{lemma}

\begin{lemma}
\label{2}
For any $\epsilon>0$, there is a $\delta_1(\epsilon)$ such that for any $0<\delta\leq\delta_1$ if $\gamma$ is $C^{k+2}$ $\delta$-close to a geodesic, then 
\begin{equation}
\int_{\gamma} (\kappa^{(k+1)})^2 ds< \frac{\epsilon^2}{3\pi(k+1)^2}.
\end{equation}
\end{lemma}

Fix $\delta_0$ as in Lemma~\ref{Poin1}, choose $\delta=\delta_0$, and let $\epsilon(\delta_0)=\epsilon_0$ be the constant chosen from Lemma~\ref{1}. For any integer $k$, choose $\epsilon_k\leq\epsilon_0$ so that 
\begin{equation}
\label{epsilonk}
\frac{1}{\epsilon_k}\geq \max\{7+2k,\sum_{\substack{{i+j+r=k} \\0\leq i,j,r\leq k}}C_{ijr}\} 
\end{equation}
where $C_{ijr}$'s are the constants appearing in (\ref{intk1}).

By Lemma~\ref{2}, for any $0<\epsilon\leq\epsilon_k$, there is a $\delta_1=\delta_1(\epsilon)$ such that for any $0<\delta\leq\delta_1$, if $\gamma$ is $C^{k+2}$ $\delta$-close to a great circle, then
\begin{equation}
\label{intkb}
\int_{\gamma} (\kappa^{(k+1)})^2 ds< \frac{\epsilon^2}{3\pi(k+1)^2}.
\end{equation}
Note that we can assume, while evolving $\gamma$ by CSF, $\gamma_t$ stays $C^{k+2}$ $\delta$-close to a great circle for at least short time, say $[0,t_1)$.
 \begin{lemma}
\label{curvedecaythm}
Let $\delta=\delta(\epsilon)\leq\delta_1$ be chosen as above. Suppose $\gamma_t$ is $C^{k+2}$ $\delta$-close to a geodesic for all $t\in [0,t_1)$, then
\begin{eqnarray}
||\kappa||_{C^k}\leq \epsilon e^{-t}
\end{eqnarray}
for all $t\in [0,t_1)$. 

\end{lemma}
\begin{remark}
\label{supl2}
For any $0\leq i\leq k$, since $\int \kappa^{(i)}=0$ and $\int (\kappa^{(i)})^2\leq \int (\kappa^{(i+1)})^2$, we have
\begin{equation}
\label{ksupl1}
\sup (\kappa^{(i)})^2\leq (\inf|\kappa^{(i)}|+\int|\kappa^{(i+1)}|)^2\leq 3\pi\int (\kappa^{({i+1})})^2\leq 3\pi\int (\kappa^{({k+1})})^2.
\end{equation}
This implies that
\begin{equation}
\label{ksupl2}
||\kappa||_{C^k}\leq \sqrt{3\pi}\sum_{i=1}^{k+1}||\kappa^{(i)}||_{L^2}\leq \sqrt{3\pi}(k+1)||\kappa^{(k+1)}||_{L^2}.
\end{equation}
Hence to prove Lemma~\ref{curvedecaythm} it is enough to show 
\begin{equation}
\label{kdecay1}
||\kappa^{(k+1)}||^2_{L^2}\leq \frac{\epsilon^2}{3\pi(k+1)^2} e^{-2t}.
\end{equation}
\end{remark}

\begin{proof}[Proof of Lemma~\ref{curvedecaythm}]
Recall \eqref{intk1},
\begin{equation}
\label{highorder}
\begin{split}
&\frac{\partial }{\partial t} \int (\kappa^{(k+1)})^2ds\\
=&\int -2(\kappa^{(k+2)})^2ds+2\int (\kappa^{(k+1)})^2ds+\int (5+2(k+1))(\kappa^2)(\kappa^{(k+1)})^2ds\\
&+\sum_{\substack{{i+j+r=k} \\0\leq i,j,r\leq k}}2C_{ijr}\int \kappa^{(i)}\kappa^{(j)}\kappa^{(r)}\kappa^{(k+1)} ds.
\end{split}
\end{equation}

Because we assume that $\gamma_t$ is $C^{k+2}$ $\delta$-close to a geodesic for all $t\in [0,t_1)$, (\ref{intkb}) holds for all $t\in [0,t_1)$. Then by \eqref{ksupl1} ${\sup (\kappa^{(i)})^2~\leq~\frac{\epsilon^2}{(k+1)^2}}$, for every $0\leq i\leq k$ and for all $t\in [0,t_1)$. Note that this also implies $\int_{\gamma} (\kappa^{(1)})^2 ds<\epsilon_0^2$ and $\gamma$ is $C^2$ $\delta_0$-close to a great circle (hence $(\ref{knC})$ holds) for all $t\in [0,t_1)$.

By (\ref{epsilonk}) and Peter-Paul inequality,
\begin{eqnarray}
\label{highorder1}
\begin{split}
&\sum_{\substack{{i+j+r=k} \\0\leq i,j,r\leq k}}2C_{ijr} \int \kappa^{(i)}\kappa^{(j)}\kappa^{(r)}\kappa^{(k+1)} ds\\
\leq&\sum_{\substack{{i+j+r=k} \\0\leq i,j,r\leq k}}2C_{ijr}\Big(\frac{1}{2\epsilon^2}\int (\kappa^{(i)}\kappa^{(j)}\kappa^{(r)})^2ds+\frac{\epsilon^2}{2}\int(\kappa^{(k+1)})^2ds\Big)\\
\leq&\sum_{\substack{{i+j+r=k} \\0\leq i,j,r\leq k}}C_{ijr}\Big(\epsilon^2\int (\kappa^{(k)})^2ds+\epsilon^2\int(\kappa^{(k+1)})^2ds\Big)\\
\leq&2\epsilon \int (\kappa^{(k+1)})^2 ds.
\end{split}
\end{eqnarray}
Together with (\ref{knC}), (\ref{highorder}) becomes 
\begin{eqnarray}
\begin{split}
\frac{\partial }{\partial t} \int (\kappa^{(k+1)})^2ds\leq -2\int (\kappa^{(k+1)})^2ds
\end{split}
\end{eqnarray}
Therefore we can conclude that as long as \eqref{intkb} holds for all $t\in[0,t_1)$, we have $||\kappa^{(k+1)}||^2_{L^2}\leq \epsilon^2 e^{-2t}/(3\pi(k+1)^2)$ for all $t\in [0,t_1)$. 
\end{proof}

\begin{proposition}
\label{he0}
Choose $\delta>0$ as in Lemma~\ref{curvedecaythm}. For any $0<\beta_0<\delta$, if $\gamma_0$ is $C^{k+2}$ $\beta_0$-close to a great circle, then $\gamma_t$ is $C^{k+2}$ $\beta_0$-close to a great circle for all $t\;\in[0,\infty)$.
\end{proposition}

\begin{proof}
For any $0<\beta_0<\delta$, there exits a $\alpha_0=\alpha_0(\beta_0)>0$ such that if $\int_{\gamma} (\kappa^{(k+1)})^2<\alpha_0$, then $\gamma$ is $C^{k+2}$ $\beta_0$-close to a great circle (Remark~\ref{1}). Choose $\gamma_0$ so that $\int_{\gamma_0} (\kappa^{(k+1)})^2<\alpha_0$. Suppose $[0,T)$ is the maximal interval on which $\gamma_t$ is $C^{k+2}$ $\beta_0$-close to a great circle $\gamma_g$. In Lemma~\ref{curvedecaythm}, we proved that $||\kappa^{(k+1)}||^2_{L^2}\leq \alpha_0 e^{-2t}$  for all $t\in [0,T)$. This implies that at $t=T$,  $||\kappa^{(k+1)}||^2_{L^2}<\alpha_0$. Therefore $\gamma_T$ is $C^{k+2}$ $\beta_0$-close to a great circle. Contradiction. 
\end{proof}

\begin{proof}[Proof of Theorem~\ref{EX}] 
For any integer $k$, let $\epsilon_k$ be the constant chosen in \eqref{epsilonk}. For every $0<\epsilon\leq\epsilon_k$, let $\delta_1(\epsilon)$ be the constant chosen in Lemma~\ref{curvedecaythm}. For any $0<\delta\leq\delta_1$, if $\gamma_0$ is $C^{k+2}$ $\delta$-close to a geodesic (i.e. $||h(\cdot, 0)||_{C^{k+2}}\leq\delta$), then $\gamma_t$ stays $C^{k+2}$ $\delta$-close to a geodesic for all t (Proposition~\ref{he0}). Hence, 
\begin{eqnarray*}
\label{theorem31}
\begin{split}
||h(\theta,t)||_{C^k}&\leq \int_t^\infty||\frac{\partial}{\partial t}h(\theta,t)||_{C^k}\\
&\leq\int_t^\infty||\frac{\sqrt{1+h^2+(h_\theta)^2}}{1+h^2}\kappa(s(\theta),t)||_{C^k}\;\;\text{(by (\ref{h}))}\\
&\leq \int_t^{\infty}(1+\delta)|| \kappa(s,t)||_{C^k}\;\;\text{(by Lemma~\ref{dd'} and Lemma~\ref{1})}\\
&\leq 2\epsilon e^{-t}\;\; \text{(by Lemma~\ref{curvedecaythm})}
\end{split}
\end{eqnarray*}
\end{proof}

\subsection{Flowing a family of curves on $(\mathbb{S}^2,g)$ by CSF}
\label{ffc}
For any curve $\gamma_0$ that divides the surface area into two equal pieces, consider a two-parameter family of curves given by the map $\Gamma:\mathbb{S}^1\times U\to\mathbb{S}^2$, where $U\subset \R^2$ is open, and $\Gamma(\cdot,0)=\gamma_0$. Moreover, for every $\xi\in U$, we assume that the curve $\Gamma(\cdot,\xi)$ divides the surface area into equal pieces. Let $\Gamma^t=\{\gamma_t|\gamma \in\Gamma\}$ where $\gamma_t$ is the $t$-time evolution by CSF. For every $k$ and any $\delta$, if we wait for long enough, we can assume that $\gamma_0$ is $C^{k+2}$ $\delta$-close to its limit $\gamma_g$. Let $D=(-\tau,\tau)\times(-\eta,\eta)$ be an open set contained in $U$ such that the curves represented by points in $D$ stay in a tubular neighborhood of $\gamma_g$ for large $t$. Define $H:[0,2\pi]\times (-\tau,\tau)\times(-\eta,\eta)\to\R$ such that $\sigma_g(x,H(x,\tau_0,\eta_0))=\Gamma(x,\tau_0,\eta_0)$ and denote the corresponding set of evolving maps by $H^t$. We proved in Theorem~\ref{EX} that every curve converges uniformly exponentially to a geodesic in the $C^k$ norm, i.e. the $k$-th derivative of $H^t$ in $\partial_x$ direction, $\partial_x^k H^t$, converges uniformly exponentially as $t\to\infty$. Next we will show that $\partial_\tau^k H^t$ and $\partial_\eta^k H^t$ also converge uniformly exponentially as $t\to\infty$. 
\begin{lemma}
\label{EX1}
Suppose $H^t$ is defined as above. For any integer $m\geq 0$,  $\partial_\tau^m H^t$ converges uniformly exponentially as $t\to\infty$.
\end{lemma}
Following by the same proof in Lemma~\ref{EX1}, one gets that $\partial_\eta^m H^t$ converges uniformly exponentially as $t\to\infty$. Then by Lemma~\ref{analysis1}, we can conclude the following Theorem: 
\begin{theorem}
\label{EX2}
Suppose $H^t$ is defined as above. For any integer $k\geq 0$, the family of maps $\partial_\xi^k H^t$ converges uniformly as $t\to\infty$ for every $\xi\in D$. 
 \end{theorem}

We consider $H$ as a variation and $H(x,0,0)=u(x)$. To compute the variation in the $\tau$-direction, we let $h(x,\tau)=H(x,\tau,0)$ and write
\begin{equation}
 \label{initialh}
 h(x,\tau)=u_0+\tau v_0+\tau^2 (w_2)_0+...+\tau^m (w_m)_0+O(\tau^{m+1}).
 \end{equation}
In the local coordinates, this one-parameter family of maps evolves by 
\begin{equation}
\label{h1}
h_t=\frac{(1+h^2)^2}{1+h^2+(h_x)^2}(h_{xx}+h),
\end{equation} 
and we denote the solution by 
\begin{eqnarray*}
h_{\tau}\equiv u(x,t)+\tau v(x,t)+\tau^2 w_2(x,t)+...+\tau^m w_m(x,t)+O(\tau^{m+1}).
\end{eqnarray*} 
The evolution equations of $u$, $v$, $w_i$, $2\leq i\leq m$ can be derived by the following:
\begin{equation}
\label{hn}
\frac{\partial^j}{\partial
\tau^j}\bigg|_{\tau=0}(h_{\tau})_t=\frac{\partial^j}{\partial
\tau^j}\bigg|_{\tau=0}\frac{(1+h_{\tau}^2)^2}{1+h_{\tau}^2+((h_{\tau})_x)^2}({(h_{\tau})}_{xx}+h_{\tau}),
\end{equation}
for $j=0,1,2,...,m$.

Note that the $w_i$'s and $v$ are $2\pi$-periodic functions in $x$, and that $\int_0^{2\pi} v^{(n)} dx=0$ and $\int_0^{2\pi} w_i^{(n)} dx=0$ for every $n\geq 0$, at any $t\geq 0$ (Lemma~\ref{intkiszero}). Therefore, by Poincar\'e inequality (Lemma~\ref{Pinequalityonedim}), we have that for every $0\leq \ell\leq n$,
\begin{equation}
\label{mn}
\int (v^{(\ell)})^2 \leq \int(v^{(n+1)})^2,
\end{equation}
\begin{equation}
\label{wmn}
\int (w_i^{(\ell)})^2 \leq \int(w_i^{(n+1)})^2,\;\;2\leq i\leq m.
\end{equation}

To prove Lemma~\ref{EX1}, we need to show that for every $0\leq j\leq m$ the solution to (\ref{hn}) converges uniformly exponentially as $t\to\infty$ in the $C^0$ norm. For any given integer $m\geq 0$, we choose $k\geq 5m+1$ and let $\epsilon_k$ be the constant chosen in Theorem~\ref{EX}. In the sequel, we will always assume that $m$, $k$ and $\epsilon_k$ have been chosen in this way.  Our proof of Lemma~\ref{EX1} is technical but elementary. We will break the proof into several lemmas and summarize the results in \S 3.2.3. 
\subsubsection{Linearized curve shortening equation}
\label{LCSFequation}
We begin by analyzing the solution to the linearized equation of (\ref{h1}) at $u$:
\begin{equation}
\label{L}
\begin{split}
v_t=&\frac{(1+u^2)^2}{1+u^2+u_x^2}v_{xx}-\frac{2u_x(1+u^2)^2(u+u_{xx})}{(1+u^2+u_x^2)^2}v_x\\
&+\frac{(1 + u^2) (1 + 4 u^2 + 3 u^4 + 2 u (1 + u^2) u_{xx} +u_x^2 (1 + 5 u^2 + 4 u u_{xx})}{(1 + u^2 + u_x^2)^2}v\\
:=&(1+a(x,t))v_{xx}+b(x,t) v_x+(1+c(x,t))v.
 \end{split}
\end{equation}

\begin{remark}A special case is when $a$, $b$, and $c$ are all zeros, equation (\ref{L}) becomes $v_t=v_{xx}+v$. This corresponds to the linearized CSF at a great circle. 
\end{remark} 
In the sequel, we assume the following condition holds. (We can make this assumption because of Theorem~\ref{EX}.)

\begin{condition}
\label{abc}
The $C^k$ norm of $a$, $b$ and $c$ are less than $\epsilon_k e^{-t}$ for all $t\geq 0$.
 \end{condition}

\begin{lemma}
\label{C1}
There is a constant $C=C(k,u_0,v_0)$ such that the solution $v(x,t)$ to (\ref{L}) satisfies
\begin{equation}
\label{mainbound}
||\frac{\partial}{\partial t}v(\cdot,t)||_{C^{k-6}} \leq Ce^{-t}
\end{equation}
for all $t\geq 0$.
\end{lemma} 
Note that because $v$ satisfies (\ref{mn}), and $\int v^{(n)}=0$, it is enough to bound $||v^{(k+1)}||_{L^2}$ instead of $||v||_{C^k}$ (Remark~\ref{supl2}). The same applies to the solutions of the higher order variation equations.
 \begin{proof}[Proof of Lemma~\ref{C1}]
We prove this Lemma by showing that
\begin{enumerate}
\item $||v(\cdot,t)||_{C^{k-1}}\leq C$ 
\item $||(v_{xx}+v)^{(k-5)}||^2_{L^2}\leq C e^{-6t}$
\end{enumerate}
for some constant $C=C(k,u_0,v_0)$. These steps will be proved in Lemma~\ref{Q} and Lemma~\ref{vxx+vexp} respectively.
Then together with Condition~\ref{abc}, and equation (\ref{L}), we have (\ref{mainbound}).
\end{proof}

\begin{lemma}\label{claimv} There is a constant $C_1=C_1(k, u_0)$ such that
\begin{equation}
\label{vne}
\frac{\partial}{\partial t}\|v^{(k)}\|^2_{L^2}
\leq C_1e^{-t} ||v^{(k)}||_{L^2}^2,
\end{equation}
for all $t\geq 0.$
\end{lemma}

\begin{proof}
The evolution equation of $\|v^{(k)}\|_{L^2}^2$:
\begin{equation}
\label{v1}
\begin{split}
&\frac{1}{2}\frac{\partial}{\partial t}\int (v^{(k)})^2 dx\\
=&\int\Big((1+a(x,t))v^{(2)}+b(x,t)v^{(1)}+(1+c(x,t))v\Big)^{(k)}v^{(k)} dx\;\;\text{ (by (\ref{L}))}\\
=&\int(1+a(x,t))v^{(k+2)}v^{(k)}dx+\int b(x,t)v^{(k+1)}v^{(k)}dx+\int(1+c(x,t))v^{(k)} v^{(k)}dx\\
&+ \sum_{\ell=0}^{k-1} {k \choose  \ell}\Big(\int a(x,t)^{(k-\ell)}v^{(\ell+2)}v^{(k)}dx+\int b(x,t)^{(k-\ell)}v^{(\ell+1)}v^{(k)}dx\\
&+\int c(x,t)^{(k-\ell)}v^{(\ell)}v^{(k)}dx\Big).
\end{split}
\end{equation}
Note that $(fg)^{(k)}=\sum_{\ell=0}^{k} {k\choose \ell} f^{(k-\ell)} g^{(\ell)}$, and $\max_\ell {k \choose \ell}={k \choose  \lfloor\frac{k}{2}\rfloor}<2^k$.

Using integration by parts and Peter-Paul inequality the first term in the last equality in (\ref{v1}) satisfies
\begin{eqnarray}
\begin{split}
&\int(1+a(x,t))v^{(k+2)}v^{(k)}dx\\
=& -\int[(1+a(x,t))v^{(k)}]^{(1)}v^{(k+1)}dx\\
\leq&-\int(1+a(x,t))(v^{(k+1)})^2dx+\frac{1}{2}\int |a(x,t)^{(1)}|(v^{(k)})^2dx\\
&+\frac{1}{2}\int |a(x,t)^{(1)}|(v^{(k+1)})^2dx.
\end{split}
\end{eqnarray}
Applying Peter-Paul inequality to the rest of the terms  in (\ref{v1}) and by (\ref{mn}), we get
\begin{eqnarray*}
\label{417}
\begin{split}
\frac{1}{2}\frac{\partial}{\partial t}\|v^{(k)}\|^2_{L^2}\leq&(-1+||a||_{C^1}+\frac{1}{2}||b||_{C^0}+\frac{1}{2}{k \choose \lfloor\frac{k}{2}\rfloor}||a||_{C^{k}})\int (v^{(k+1)})^2 dx \\
&+(1+\frac{1}{2}||a||_{C^1}+\frac{1}{2}||b||_{C^0}+||c||_{C^0})\int (v^{(k)})^2dx\\
&+{k \choose \lfloor\frac{k}{2}\rfloor}(\frac{1}{2}||a||_{C^{k}}+||b||_{C^k}+||c||_{C^k})\int (v^{(k)})^2dx.
\end{split}
\end{eqnarray*}
Note that if necessary one can replace the restriction in (\ref{epsilonk}) by $$\frac{1}{3\epsilon_k}\geq \max\{7+2k,\sum_{\substack{{i+j+r=k} \\0\leq i,j,r\leq k}}C_{ijr},{k \choose \lfloor\frac{k}{2}\rfloor}\}$$ 
so that the coefficient of $\int (v^{(k+1)})^2 dx$ is less than zero for all $t\geq 0$. Using (\ref{mn}), we can find a constant $C_1=C_1(k,u_0)$ such that
 \begin{eqnarray}
 \label{boundforv}
 \begin{split}
\frac{\partial}{\partial t}\|v^{(k)}\|^2_{L^2}
\leq&2 (2||a||_{C^1}+||b||_{C^0}+\frac{1}{2}||c||_{C^0})\int (v^{(k)})^2dx\\
&+2{k \choose \lfloor\frac{k}{2}\rfloor}(||a||_{C^{k}}+||b||_{C^k}+||c||_{C^k})\int (v^{(k)})^2dx\\
\leq&C_1e^{-t} ||v^{(k)}||_{L^2}^2.
\end{split}
\end{eqnarray}
\end{proof}

\begin{lemma}
\label{Q}There is a constant $C=C(k,u_0,v_0)$ such that $v(x,t)$ satisfies
\begin{equation}
||v(\cdot,t)||_{C^{k-1}}\leq C
\end{equation}
for all $t\geq 0$. 
\end{lemma}
\begin{proof}
By Lemma~\ref{claimv} and Lemma~\ref{Gronwall},
\begin{equation}
\|v^{(k)}\|^2_{L^2}\leq \|v^{(k)}(\cdot,0)\|_{L^2}^2 e^{C_1-C_1e^{-t} }\leq\|v^{(k)}(\cdot,0)\|_{L^2}^2 e^{C_1}
\end{equation}
for all $t\geq 0$. 
\end{proof}

\begin{lemma}
\label{vxx+vexp}
There is a constant $C=C(k,u_0,v_0)$ such that $v(x,t)$ satisfies $$||(v_{xx}+v)^{(k-5)}||_{L^2}^2\leq C e^{-6t}$$ for all $t\geq 0$. 
\end{lemma}
\label{vp}
\begin{remark}If $f$ is a smooth $2\pi$-periodic function and $\int_0^{2\pi} f^{(n)}(x)dx =0$ for all $n$, then
\begin{eqnarray}
\label{vp1}
||(f_{xx}+f)^{(n)}||_{L^2}^2\leq \frac{1}{4}||(f_{xx}+f)^{(n+1)}||_{L^2}^2.
\end{eqnarray}
Note that $u$, $v$, and $w_i$'s satisfy this.
\end{remark}

\begin{proof}[Proof of Lemma~\ref{vxx+vexp}]
We study the time evolution equation for $||(v_{xx}+v)^{(k-5)}||_{L^2}^2$.
 \begin{eqnarray*}
 \label{v2+v}
 \begin{split}
 &\frac{\partial}{\partial t}\int((v_{xx}+v)^{(k-5)})^2dx\\
=& -2\int ((v^{(2)}+v)^{(k-4)})^2dx+2\int ((v^{(2)}+v)^{(k-5)})^2dx\\
&+2\int ((av^{(2)}+bv^{(1)}+cv)^{(k-3)}+(av^{(2)}+bv^{(1)}+cv)^{(k-5)}) (v^{(2)}+v)^{(k-5)}   dx.
\end{split}
 \end{eqnarray*}
By (\ref{mn}), Lemma~\ref{Q} and Condition~\ref{abc}, there is a constant $K=K(k,u_0,v_0)$ such that
\begin{eqnarray*}
 \label{Bv}
 \begin{split}
&2\int ((av^{(2)}+bv^{(1)}+cv)^{(k-3)}+(av^{(2)}+bv^{(1)}+cv)^{(k-5)})(v^{(k-3)}+v^{(k-5)})dx\leq K e^{-t}.
\end{split}
 \end{eqnarray*}
In addition, by Remark~\ref{vp1},
 \begin{eqnarray}
 \label{v2+vs}
 \begin{split}
\frac{\partial}{\partial t}\int((v_{xx}+v)^{(k-5)})^2dx\leq -6\int ((v^{(2)}+v)^{(k-5)})^2dx+K e^{-t}.\\
\end{split}
\end{eqnarray}
Thus Lemma~\ref{Gronwall} implies that
\begin{eqnarray*}
\label{FV}
\begin{split}
|| (v_{xx}+v)^{(k-5)})||^2_{L^2} \leq & e^{-6t}(||(v_{xx}+v)^{(k-5)}(\cdot,0)||^2_{L^2}+K-Ke^{-t})\leq C e^{-6t} .
\end{split}
\end{eqnarray*} 

\end{proof}

\subsubsection{Higher order variation equations}
In this section, we study the higher order evolution equations. A computation shows that the evolution equation of $w_i$ for $2\leq i\leq m$ satisfies  
\begin{equation}
\begin{split}
\label{ew}
(w_i)_t&=(w_i)_{xx}+w_i+a(x,t) (w_i)_{xx}+b(x,t)(w_i)_x+c(x,t)w_i+d_{i}(x,t)
\end{split}
\end{equation}
where $a$, $b$ and $c$ are defined in (\ref{L}), and
\begin{eqnarray}
\label{di}
\begin{split}
d_2&=(u+u_{xx})U_{2}+(v + v_{xx})V_{1},\\
d_{i}&=(u+u_{xx})U_{i}+(v+v_{xx})V_{i-1}+\sum_{j=2}^{i-1}(w^{(2)}_{j}+w_{j} )V_{i-j},\;\;i\geq 3
\end{split}
\end{eqnarray}
where 
\begin{eqnarray*}
\begin{split}
U_2=&\frac{4 u^2 v^2 + 2 (1 + u^2) v^2 }{(
    1 + u^2 + u_x^2)} - \frac{
    4 u (1 + u^2) v (2 u v + 2 u_x v_x)}{(1 + u^2 + 
      u_x^2)^2}\\ &+ \Big( \frac{
        (2 u v + 2 u_x v_x)^2}{(1 + u^2 + u_x^2)^3} - \frac{
        v^2 +  v_x^2}{(1 + u^2 + u_x^2)^2}\Big)(1+u^2)^2\\
V_1=&\frac{4u (1 + u^2) v}{1 + u^2 + u_x^2} - \frac{(1 + u^2)^2 (2 u v + 2 u_x v_x)}{(1 + u^2 + u_x^2)^2},
\end{split}
\end{eqnarray*}
and $U_i$, $V_{i-j}$ are functions of $u$, $v$, $w_j$ for $2\leq j\leq i-1$ and their first and second partial derivatives in the $x$ direction. They can be derived as follows:
\begin{eqnarray*}
\begin{split}
&U_{i}=\frac{1}{i!}\sum_{j=1}^{i-1}{i \choose j}\Big(\frac{\partial^{i-j}}{\partial
\tau^{i-j}}\bigg|_{\tau=0}(1+h_{\tau}^2)^2\Big)\Big(\frac{\partial^j}{\partial\tau^j}\bigg|_{\tau=0}((1+h_{\tau}^2+((h_{\tau})_x)^{2})^{-1}\Big),\\
&V_{i-j}=\frac{j!}{i!}\frac{\partial^{i-j}}{\partial
\tau^{i-j}}\bigg|_{\tau=0}\frac{(1+h_{\tau}^2)^2}{1+h_{\tau}^2+((h_{\tau})_x)^2}. 
\end{split}
\end{eqnarray*}

\begin{lemma}
\label{Cn}
There is a constant $C=C(k,u_0,v_0,(w_2)_0...,(w_i)_0)$ such that the solution $w_i(x,t)$ satisfies
\begin{equation}
\label{mainboundw}
||\frac{\partial}{\partial t}w_i(\cdot,t)||_{C^{k-1-5i}} \leq Ce^{-t}
\end{equation}
for all $t\geq 0$.
\end{lemma} 

We will show inductively that the following are satisfied.
\begin{enumerate}
\item $||d_{i}^{(k-5(i-1)))}||_{L^2}^2\leq C e^{-6t}$
\item $||w_i(\cdot,t)||_{C^{k-5(i-1)-1}}\leq C_1$
\item $||((w_{i})_{xx}+w_i)^{(k-5i)}||_{L^2}^2\leq C_2e^{-6t}$
\end{enumerate}
for constants $C$, $C_1$ and $C_2$ depending on $k$, and initial values $u_0$, $v_0$ and $(w_\ell)_0$, for all $\ell\leq i$. The Lemma follows immediately from these three inequalities.

We have the following estimate for $d_2$:
\begin{lemma}
\label{C0} There is a constant $C=C(k,u_0,v_0)$ such that 
\begin{equation}
\label{dn}
||d_2^{(k-5)}||_{L^2}^2=||((u_{xx}+u)U_2)^{(k-5)}+((v_{xx}+v)V_1)^{(k-5)}||_{L^2}^2\leq C e^{-6t}.
\end{equation}
\end{lemma}

\begin{proof}
Because $||U_2||_{C^{k-5}}$ and $||V_1||_{C^{k-5}}$ are both bounded by a constant and we proved in Lemma~\ref{vxx+vexp} that $||(v_{xx}+v)^{(k-5)}||^2_{L^2}\leq Ce^{-6t}$, it is enough to show that 
\begin{equation}
||(u_{xx}+u)^{(k-5)}||^2_{L^2}\leq C e^{-6t}.
\end{equation}
Recall that 
\begin{equation}
\label{Lu}
u_t=\frac{(1+u^2)^2}{1+u^2+u_{x}^2}(u_{xx}+u)\equiv (1+a_0)(u_{xx}+u),
\end{equation}
where $||u||_{C^k}\leq \epsilon_k e^{-t}$ and $||a_0||_{C^k}\leq \epsilon_ke^{-t}$. The result follows by the same argument in Lemma~\ref{vxx+vexp} with $a=a_0$, $b=0$ and $c=0$.
\end{proof}
With the estimate for $d_2$, we are now ready to derive the estimate for $w_2$.
 \begin{lemma}
 \label{ewl}
 There is a constant $C=C(k,u_0,v_0,(w_2)_0)$ such that $w_2$ satisfies $$||w_2||_{C^{k-6}}\leq C$$
 for all $t\geq 0$. 

 \end{lemma}
\begin{proof}
Recall
\begin{eqnarray*}
\begin{split}
(w_2)_t&=(w_2)_{xx}+w_2+a(x,t) (w_2)_{xx}+b(x,t)(w_2)_x+c(x,t)w_2+d_{2}(x,t).
\end{split}
\end{eqnarray*}
Note that $a$, $b$ and $c$ satisfy the Condition~\ref{abc} and $w_2$ satisfies (\ref{wmn}). Therefore, one can use the same method as in Lemma~\ref{Q} to derive $L^2$ estimates for $(w_2)_t-d_2$. This gives
\begin{eqnarray}
\label{ww}
\begin{split}
\frac{\partial}{\partial t}\|w_2^{(k-5)}\|^2_{L^2}\leq& Ke^{-t}||w_2^{(k-5)}||_{L^2}^2+2\int d_2^{(k-5)}w_2^{(k-5)}, 
\end{split}
\end{eqnarray} 
where $K=K(k,u_0,v_0,(w_2)_0,...,(w_i)_0)$.

By (\ref{dn}) and Peter-Paul inequality,
\begin{eqnarray}
\begin{split}
\int d_2^{(k-5)}w_2^{(k-5)} &\leq e^{2t}\int( d_2^{(k-5)})^2+e^{-2t} \int (w_2^{(k-5)})^2\\
&\leq Ce^{-4t}+e^{-2t}||w_2^{(k-5)}||_{L^2}^2.
\end{split}
\end{eqnarray}
Hence
 \begin{equation}
\label{forw}
 \begin{split}
\frac{\partial}{\partial t}\|w_2^{(k-5)}\|^2_{L^2}
\leq& K e^{-t}||w_2^{(k-5)}||_{L^2}^2+2Ce^{-4t}+2e^{-2t}||w_2^{(k-5)}||_{L^2}^2.
\end{split}
\end{equation}
Lemma~\ref{ewl} follows by Lemma~\ref{Gronwall}.  

\end{proof}

Next, we repeat the same argument as in Lemma~\ref{vxx+vexp} to derive the following Lemma. The only difference is that there is an extra term contributed by $d_2$.

\begin{lemma}
There is a constant $C=C(k,u_0,v_0,(w_2)_0)$ such that
\label{wxx+wt}
\begin{eqnarray}
\label{wieq}
\begin{split}
||((w_{2})_{xx}+w_{2})^{(k-10)}||^2_{L^2}\leq& Ce^{-6t},
\end{split}
\end{eqnarray}
for all $t\geq 0$.
\end{lemma}

\begin{proof}
We can apply the same estimate derived in (\ref{v2+vs}) to $(w_2)_t-d_2$. In Lemma~\ref{vxx+vexp}, we used that $||v||_{C^{k-1}}$ is bounded by a constant. In this case, we have $||w_2||_{C^{k-6}}<C$. Hence, we get an estimate only up to order $k-10$.
\begin{eqnarray*}
\label{ewm}
\begin{split}
&\frac{\partial}{\partial t}\int(((w_{2})_{xx}+w_{2})^{(k-10)})^2 dx\\
\leq& -6\int((w_{2})_{xx}+w_{2})^{(k-10)})^2 dx+Ke^{-t}\\
&+2\int((w_{2})_{xx}+w_{2})^{(k-10)}(d_2^{(2)}+d_2)^{(k-10)}dx.
\end{split}
\end{eqnarray*}
From (\ref{dn}) and Lemma~\ref{ewl},
\begin{eqnarray}
\int((w_{2})_{xx}+w_{2})^{(k-10)}(d_2^{(2)}+d_2)^{(k-10)}dx \leq Ce^{-3t}.
\end{eqnarray}
Hence
\begin{equation}
\label{w3}
\begin{split}
\frac{\partial}{\partial t}\int(((w_{2})_{xx}+w_{2})^{(k-10)})^2 dx\leq-6\int((w_{2})_{xx}+w_{2})^{(k-10)})^2 dx+Ce^{-3t}+Ke^{-t}.
\end{split}
\end{equation}
The Lemma follows by Lemma~\ref{Gronwall}.
\end{proof}

\begin{proof}[Proof of Lemma~\ref{Cn}]
When $i=2$ we showed
\begin{enumerate}
\item $||d_{2}^{(k-5)}||_{L^2}^2\leq C e^{-6t}$
\item $||w_2(\cdot,t)||_{C^{k-6}}\leq C_1$
\item $||((w_{2})_{xx}+w_2)^{(k-10)}||_{L^2}^2\leq C_2e^{-6t}$
\end{enumerate}
in Lemma~\ref{C0}, Lemma~\ref{ewl}, and Lemma~\ref{wxx+wt} respectively.
Together with Condition~\ref{abc}, we conclude that there is a constant $C=C(k,u_0,v_0,(w_2)_0...,(w_i)_0)$ such that
\begin{equation}
\label{mainboundw}
||\frac{\partial}{\partial t}w_2(\cdot,t)||_{C^{k-11}} \leq C e^{-t}
\end{equation}
for all $t\geq 0$.
From (\ref{di}), and Lemma~\ref{wxx+wt} we have
\begin{equation}
\label{d2}
||d_3^{(k-10)}||_{L^2}^2\leq Ce^{-6t},
\end{equation}
for all $t\geq 0$.
Inductively, for every $3\leq i\leq m$, one may apply the same arguments in Lemma~\ref{ewl} and in Lemma~\ref{wxx+wt} to get the result.
\end{proof}
\subsubsection{Proof of Lemma~\ref{EX1}}

\begin{proof}
For any integer $m\geq 0$, we let $k=5m+1$. By Theorem~\ref{EX}, Lemma~\ref{C1} and Lemma~\ref{Cn}, we have that $\partial_\tau^i H^t$ converges uniformly exponentially in the $C^{0}$ norm for all $0\leq i\leq m$. 
\end{proof}

\subsection{A family of smooth projective planes}
\label{homotopy}
In this section, we will use CSF to construct a family of smooth projective planes.
We first prove a proposition which provides a sufficient condition for a family of curves to define a smooth projective plane. 
We begin with a tuple $(\P,\L,\F,\pi_{\P},\pi_{\L})$ which satisfies the following: $(\P,\L,\F)$ is a projective space.
The set of points $\P$, and the set of lines $\L$ are closed smooth manifolds of dimension $2$, and the flag space $\F\subset \P\times\L$ is a closed smooth manifold of dimension $3$. We call $p\in\P$ on a line $\ell\in\L$ if $(p,\ell)\in F$ and we define the projections by $\pi_\P:\F\to\P: (p,\ell) \mapsto p$ and $\pi_\L:\F\to\L: (p,\ell)\mapsto \ell$. 

Suppose $\pi_\L$ is a submersion and $\pi_\P$ is a smooth map. For any $\ell\in\L$, we denote $\hat{\ell}=\pi^{-1}_\L(\ell)=\{(p,\ell)|p\in P_\ell\}\subset \F$ and $\bar{\ell}=\pi_\P(\hat{\ell})\subset \P$. Note that a nonzero tangent vector $\xi\in \T_\ell\L$ corresponds to a ``variation of $\hat{\ell}$" in $\F$. To see this, consider a smooth path $s\mapsto \ell_s\in\L$ with $\ell_0=\ell$ and $\frac{\partial}{\partial s} \ell_s=\xi$; then $ \pi_\L^{-1}(\ell_s)=\hat{\ell_s} \subset \F$ is a family of curves depending on the parameter $s$. Using the fact that $\pi_\L$ is a smooth fiber bundle, we may use a trivialization of $\pi_\L$ near $\ell$ to obtain a smooth map $\tau_{\ell,\ell'}:\hat{\ell}\to\hat{\ell'}$ which depends smoothly on $\ell$, $\ell'\in\L$. Hence, $\displaystyle \hat{\ell} \xrightarrow{\tau_{\ell,\ell_s}} \hat{\ell_s} \xrightarrow{\pi_\P}\P$ determines a family of smooth curves in $\P$ depending on the parameter $s$; denote it by $\bar{\ell}_s$. Suppose the map $\pi_\P$ satisfies that for every two distinct $\ell_1$, $\ell_2\in\L$, the lines $\bar{\ell}_1$, $\bar{\ell}_2\subset\P$ intersect exactly once, and transversely. Assume further that for every $\ell\in\L$, the restriction of $\pi_\P$ to $\hat{\ell}\subset \F$ is a smooth embedding. Then for any vector $\xi\in T_\ell\L$, we call $\bar{\xi}=\frac{\partial}{\partial s}|_{s=0}\bar{\ell}_s$ the {\bf corresponding variation} of $\bar{\ell}$. 

\begin{proposition}
\label{subvec}
For any tuple $(\P,\L,\F,\pi_{\P},\pi_{\L})$ defined as above, suppose for every $\ell\in\L$ and every nonzero vector $\xi\in T_\ell\L$, the normal component of the corresponding variation $\bar{\xi}$ of $\bar{\ell}$ has precisely one zero, which is transverse. Then $\pi_\P$ is a submersion. 
\end{proposition}

\begin{proof}
Fix $\ell\in\L$, for any $p$ on $\bar{\ell}$, the tangent space $T_p\P$ is two dimensional and one can decompose it into $T_p\bar{\ell}\oplus N_p\bar{\ell}$, where $T_p\bar{\ell}$ is a space of dimension one that is tangent to $\bar{\ell}$ at $p$ and $N_p\bar{\ell}=(T_p\bar{\ell})^{\perp}$ is its orthogonal complement. We will show that $\pi_\P$ is a submersion by showing its differential is surjective onto both subspaces.

Define a map $\phi_p:T_\ell\L\to N_p \bar{\ell}$ such that any vector $\xi\in T_\ell\L$ is mapped to the normal component of corresponding variation field $\bar{\xi}$ at $p$ on $\bar{\ell}$. We will show that the map $\phi_p$ is onto. Suppose not, since $\dim N_p\bar{\ell}=1$, $\phi_p$ has to be a zero map. Choose another point $q\neq p$ on $\bar{\ell}$ and define the decomposition of the tangent space $T_q\P$ and the map $\phi_q:T_\ell\L\to N_q \bar{\ell}$ as we did for the point $p$. Since $\dim(T_\ell\L)=2>1=\dim(N_q\bar{\ell})$, the kernel of the map $\phi_q$ must be at least one dimensional. Hence there is a vector $\xi'\in T_\ell\L$ such that the normal component of the corresponding vector field $\bar{\xi}'$ along $\bar{\ell}$ vanishes at $q$, that is $\phi_q(\xi')=0$. On the other hand, since $\phi_p$ is a zero map, $\phi_p(\xi')=0$. So, we found a nonzero vector $\xi'\in T_\ell\L$ such that the normal component of the corresponding vector field $\bar{\xi}'$ along $\bar{\ell}$ has two transverse zeros, $p$ and $q$. This contradicts with our assumption that there is precisely one transverse zero, therefore $\phi_p$ is onto. In addition, since $\pi_\P|_{\hat{\ell}}$ is a smooth embedding, a nonzero tangent $\hat{v_p}\in T_{(p,\ell)}\hat{\ell}$ is mapped to a nonzero tangent $v_p\in T_p\bar{\ell}$. We therefore conclude that $\pi_\P$ is a submersion.
\end{proof}

Our next goal is to show that the CSF gives rise to a smooth homotopy, when applied to an arbitrary smooth projective plane. We will prove in a later section that this is in fact a smooth homotopy of smooth projective planes from the initial smooth projective plane to the standard $\mathbb{RP}^2$. In the sequel, we let the tuple $(\P,\L,\F,\pi^0_{\P},\pi_\L)$ be a smooth projective plane. Note that $\P$, $\L$ are diffeomorphic to $\mathbb{RP}^2$. We endow $\P$ with a Riemannian metric of constant curvature $1$, making it isometric to $\mathbb{RP}^2$ with its usual metric. We use CSF to deform $\pi^0_\P$ through a family of smooth maps by defining 
\begin{equation}
\label{pipt}
\pi_\P^t:\F\to\P
\end{equation}
to be the unique map with the property that $\pi_\P^t|_{\hat{\ell}}$ is the result of applying the CSF to $\xymatrix{\hat{\ell}\ar[r]^{\pi_\P} &\P}$ for time $t$. The solutions to the CSF exist for any time and depend smoothly on the initial conditions \cite{Huisken96}. Since $\pi_\P^0$ is smooth, this gives a smooth map $\phi:\F\times[0,\infty)\to\P$ by setting $\phi(-,t)=\pi_\P^t$. The next Proposition shows that $\lim_{t\to\infty}\pi_\P^t$ exists and is smooth, and $\phi$ induces a smooth homotopy from $\pi_\P^0$ to this limit map. 
 
 \begin{proposition}
 \label{submersion}
Let $(\P,\L,\F,\pi^0_{\P},\pi_\L)$  and $\{\pi_\P^t:\F\to\P\}_{t\in[0,\infty)}$ be as above. Then $\pi_\P^{\infty}:=\lim_{t\to\infty}\pi_\P^t$ exists and is smooth. Moreover, we can find a reparametrization $\alpha:[0,1]\to[0,\infty]$ so that $\Phi:=\phi\circ\alpha:\F\times[0,1]\to\P$ is a smooth homotopy from $\pi_\P^0$ to $\pi_\P^{\infty}$.
\end{proposition}

\begin{proof}
Since $\pi_\L$ is a submersion between compact manifolds, it is a smooth locally trivial fibration. For every $\ell\in\L$, $\pi_\L^{-1}(\ell)$ is diffeomorphic to $\mathbb{S}^1$. Let $\Lambda\subset\L$ be an open set contains $\ell$, then $\pi_\L^{-1}(\Lambda)\cong \mathbb{S}^1\times\Lambda\subset \F$.  Recall that lines of smooth projective plane are not null-homotopic (Remark~\ref{nullhomo}) and any curve that is not null-homotopic in $\P$ (diffeomorphic to $\mathbb{RP}^2$) lifts to an area-bisecting curve in $\mathbb{S}^2$. By Lemma~\ref{EX}, for each $\ell\in \Lambda$, the lift of $\pi_\P^t|_{\mathbb{S}^1\times\{\ell\}}$ converges in the $C^k$ norm to a parametrization of a great circle as $t\to\infty$, with the convergence uniform in $\ell$. The uniform convergence in $\ell$ implies that $\pi_\P^t$ converges uniformly (i.e. in the $C^0$ topology) to $\pi_\P^{\infty}$ as $t\to\infty$. Theorem~\ref{EX2} implies that for every $x\in \mathbb{S}^1$, and every $k\in \N$, the $k$-jet of $\pi_\P^t |_{\{x\}\times \Lambda}$ at any point $(x,\lambda)\in \mathbb{S}^1\times\Lambda$ converges uniformly in $\lambda$ as $t\to \infty$. By Lemma~\ref{analysis1}, we have that for any $k\geq 0$ the family of maps $\pi_\P^t|_{\pi_\L^{-1}(\Lambda)}$ converges to $\pi_\P^{\infty}|_{\pi_\L^{-1}(\Lambda)}$ in the $C^k$ topology as $t\to\infty$, so that the latter is smooth. Since $\ell\in\L$ was arbitrary, $\pi_\P^{\infty}$ is smooth. 

Let $\Phi_0:\F\times[0,1]\to\P$ be defined by $\Phi_0(-,s)=\phi(-,\frac{s}{1-s})$. Then $\Phi_0$ is a continuous homotopy from $\pi_\P^0$ to $\pi_\P^\infty$, but not necessarily a smooth one, since its left derivatives at $s=1$ need not exist. Our goal is to find a continuous reparametrization $\beta:[0,1]\to[0,1]$ so that $\Phi:=\Phi_0\circ\beta$ is the desired smooth homotopy from $\pi_\P^0$ to $\pi_\P^{\infty}$. Clearly $\beta$ will be  smooth on $[0,1)$ but not necessarily at 1.     

In what follows, we will think of $\Phi_0$ as a map from $[0,1]$ to $C^{\infty}(\F,\P)$. Further, using a smooth embedding of $\P$ into $\mathbb{R}^4$, we realize $C^\infty(\F,\P)$ as a subset of $C^{\infty}(\F,\R^4)$, which is a Fr\'{e}chet space with a Fr\'{e}chet metric, say $|\cdot|_\F$. The $C^\infty$ convergence of $\pi_\P^t$ to $\pi_\P^\infty$ (proved above) is equivalent to the continuity of $\Phi_0(s)$ at $s=1$, and the smooth dependence of $\pi_\P^t$ on $t\in[0,\infty)$ implies that $\Phi_0(s)$ is smooth on $[0,1)$, in the sense of G\^ateaux. Our goal is to find a continuous reparametrization $\beta:[0,1]\to[0,1]$, which is smooth on $[0,1)$, so that $\Phi_0(\beta(s))$ is also smooth at $s=1$, i.e. its left (G\^ateaux) derivatives at $s=1$ exist to all orders.

Clearly, it is enough to find a continuous reparametrization $\beta:[0,1]\to[0,1]$, smooth except at 1, so that $|\Phi_0(\beta(1)) - \Phi_0(\beta(s))|_\F = o( (1-s)^k)$ as $s \to 1-$, for all $k \in \N$; indeed, $\Phi_0(\beta(s))$ would then not just be smooth but also flat at $s=1$, i.e. all its left derivatives would be $0$. We construct such a $\beta$ as follows:

First, find a continuous strictly decreasing function $\rho$ on [0,1], such that $\rho$ is smooth on [0,1), $\rho(s) > |\Phi_0(1) - \Phi_0(s)|_\F$ for $0 \leq s < 1$, and $\rho(1) = 0$. Such a $\rho$ can be constructed by first assigning the values $\rho(1-\frac{1}{n}) := \frac{1}{n} + \sup_{s\in[1-\frac{1}{n-1},1]}|\Phi_0(1) - \Phi_0(s)|_\F$, for $n=2,3,4,\ldots$, and $\rho(0) := \rho(1-\frac{1}{2}) + 1$. (The sups exist since $\Phi_0(s)$ is continuous.) Next, interpolate $\rho(s)$ linearly over each interval $[1-\frac{1}{n}, 1-\frac{1}{n+1}],\; n = 1,2,3,\ldots$. This gives a piecewise-linear function that satisfies all the desired properties, except differentiability at the points $1-\frac{1}{n}, n \in \N$. Finally, ``smooth out" this piecewise-linear function at these points in such a way that the resulting function $\rho(s)$ continues to be strictly decreasing, and stays bigger than the function $|\Phi_0(1) - \Phi_0(s)|_\F$.

Define $$ \beta(s):= \rho^{-1}\Big(\rho(0)\frac{ \int_0^{1-s} e^{-u^{-2}} du }{\int_0^1 e^{-u^{-2}} du}\Big).$$

It is easy to see that $\beta(0) = 0$, $\beta(1) = 1$, $\beta:[0,1] \to [0,1]$ is continuous and strictly increasing, and $\beta$ is smooth on $[0,1)$. It only remains to check that $\Phi_0(\beta(s))$ is flat at $s = 1$, and we do this as follows:
$|\Phi_0(\beta(1)) - \Phi_0(\beta(s))|_\F = |\Phi_0(1) - \Phi_0(\beta(s))|_\F < \rho(\beta(s))$ = $\rho(0)\frac{ \int_0^{1-s} e^{-u^{-2}} du }{\int_0^1 e^{-u^{-2}} du} < \rho(0)\frac{(1-s)e^{-(1-s)^{-2}}}{\int_0^1 e^{-u^{-2}} du} = o( (1-s)^k)$ as $s \to 1-$, for all $k \in \N$, as desired.

To deduce the statement of the Proposition, define $\alpha:[0,1] \to [0,\infty]$ by $\alpha(s) =\frac{\beta(s)}{1-\beta(s)}$, so that $\Phi := \phi \circ \alpha = \Phi_0 \circ \beta$, which we have just shown to be a smooth homotopy from $\pi_\P^0$ to $\pi_\P^{\infty}$.
\end{proof}

By Proposition~\ref{submersion} and the following argument, we see that for every $t \in[0,\infty]$, the flag space $\F$ of the tuple $(\P,\L,\F,\pi^t_{\P},\pi_\L)$ is a smooth submanifold of $\P\times\L$. Define
$$
\Psi^t : (\pi_\P^t,\pi_\L): \F \to \P \times \L.
$$ 
Since both component maps are smooth, this is a smooth map. 

To see that $\Psi^t$ is one-to-one, note that two elements $(p_1,\ell_1)$, $(p_2,\ell_2)$
with the same image must in particular have the same image under $\pi_\L$, which
means that $\ell_1 = \ell_2$.   But then the restriction of $\pi_\P^t$ to a fiber
of $\pi_\L$ is a parametrization of an embedded curve in $\mathbb{RP}^2$, and is injective;
therefore $p_1 = p_2$.

To see that $\Psi^t$ is an immersion, consider the derivative of $\Psi^t$ at some $
(p,\ell)$ in $\F$. If a tangent vector lies in the kernel, of $D\Psi^t$, then it must lie in the kernel of $D\pi_\L$, which means that it is tangent to the fiber of $\pi_\L$ passing through $\ell$; but the restriction of $\pi_\P^t$ to this
fiber is a diffeomorphism onto an embedded curve, and has injective derivative.
Thus the derivative of $\Psi^t$ is injective.

Since $\Psi^t$ is an injective immersion of a compact smooth manifold into a smooth manifold, it is an embedding, i.e. $\F$ is a submanifold of $\P\times\L$.

\begin{lemma}
\label{iso1}
Let $(\P,\L,\F,\pi^0_{\P},\pi_\L)$ be a smooth projective plane. Then for each $t\in[0,\infty)$, the tuple $(\P,\L,\F,\pi_{\L},\pi^t_{\P})$ is a smooth projective plane. 
\end{lemma}
\begin{proof}
We first verify that the tuple $(\P,\L,\F,\pi_{\L},\pi^t_{\P})$ satisfies SPP1 in Definition~\ref{SP1}. For every $\ell\in\L$,  the restriction of $\pi_\P$ to $\hat{\ell}\subset\F$ is a smooth embedding and it remains smoothly embedded under CSF [see Theorem 3.1, \cite{Gage}]. We adopt the notation $\bar{\ell_t}=\pi^t_\P(\hat{\ell})\subset \P$, so for fixed $\ell\in \L$ we get a family of smooth curves in $\P$ depending on the parameter $t$. Since we begin with a smooth projective plane, at $t=0$ any two point rows $P_{\ell_1}^0:=(\bar{\ell_1})_0=\pi^0_\P(\pi_\L(\ell_1))$, $P_{\ell_2}^0:=(\bar{\ell_2})_0=\pi^0_\P(\pi_\L(\ell_2))$ intersect transversely at precisely one point. By Corollary~\ref{interc}, the point rows $P_{\ell_1}^t$ and $P_{\ell_2}^t$ remain intersecting transversely at exactly one point at any $t\in[0,\infty)$. 

Then we will use Proposition~\ref{subvec} to show that $\pi^t_\P$ is a submersion for all time. Recall that for $\ell\in\L$ and a smooth path $s\mapsto \ell_s\in\L$ with $\ell_0=\ell$, $ \pi_\L^{-1}(\ell_s)=\hat{\ell_s} \subset \F$ is a family of curves depending on the parameter $s$. For every $s$, we obtain a CSF $t\mapsto \pi_\P^t|_{\hat{\ell}_s}:\hat{\ell}_s\to \P$. Since $\pi_\L$ is a smooth fiber bundle, we can use the smooth map $\tau_{\ell,\ell'}:\hat{\ell}\to\hat{\ell'}$ (It is defined in the beginning of this section.) to adjust the domains so that for all $s$, $t$ the composition $\displaystyle \hat{\ell} \xrightarrow{\tau_{\ell,\ell_s}} \hat{\ell_s} \xrightarrow{\pi_\P^t}\P$ defines a family of CSF's depending on the parameter $s$. 

Differentiating the family of CSF's with respect to $s$, we obtain a solution to the linearized CSF, linearized at $t\mapsto\pi_\P^t|_{\hat{\ell}}$, which (at time $t$) is a vector field along $\bar{\ell}_t$. 
 
At $t=0$ the solution to the linearized flow is a vector field $\bar{\xi}$ along $\pi^0_\P|_{\hat{\ell}}:\hat{\ell}\to \P$ corresponding to the variation of $\bar{\ell}\subset \P$ defined by $\xi\in \T_\ell\L$. Recall that if $\xi\neq 0$, then the component of $\bar{\xi}$ normal to $\bar{\ell}$ is a normal vector field which vanishes precisely once, and transversely. At any $t\in[0,\infty)$, the number of transverse zeros of the solution to the LCSF cannot increase due to Proposition~\ref{transversezeros}, and it cannot decrease to zero since the normal bundle of $\bar\ell$ is a twisted line bundle, i.e. it is  a Mobius band; hence it does not have a nowhere vanishing section.
Therefore, at each $t\in[0,\infty)$, the normal component of the corresponding variation of $\bar{\ell}_t$ has precisely one transverse zero. By Proposition~\ref{subvec}, $\pi_\P^t$ is a submersion for all $t$.

What remains to check is SPP2 in Definition~\ref{SP1}. For any tuple, we denote the line pencil through $p\in\P$ by $L_p^t:=\pi_\L((\pi_\P^t)^{-1}(p))$. Any element (line) of $L^t_p$ corresponds to a point row through $p$. The line pencil $L^t_p$ corresponds to a family of point rows where any pair of point rows intersect only at $p$, and transversely. 

Moreover, any two line pencils $L_p^t$ and $L_q^t$ intersect at exactly one point $\ell$ in $\L$ because if they intersect at more than one point, then there are point rows intersecting at two points, $p$ and $q$.

Next, we will verify that $L_p^t$ and $L_q^t$ intersect transversally at $\ell$. Let $\xi^p\in T_\ell L^t_p$, then the normal component of the corresponding variation of $\bar{\ell}_t$ vanishes at exactly one point $p$. Similarly, if $\xi^q\in T_\ell L^t_q$, then the normal component of the corresponding variation of $\bar{\ell}_t$ vanishes at precisely one point $q$. Recall that for any $p\in \bar{\ell}_t$, there is a one dimensional subspace of $V_p\subset T_\ell\L$ such that $\xi\in V_p$ iff $\bar{\xi}$ is tangent to $\bar{\ell}_t$ at $p$. Since $p\neq q$ and $\dim T_\ell\L=2$, $\xi^p$ and $\xi^q$ must span the tangent space $T_\ell\L$. 

By Definition~\ref{SP1} we can conclude that  at each $t\in[0,\infty)$, the tuple $(\P,\L,\F,\pi_{\L},\pi^t_{\P})$ is a smooth projective plane.
\end{proof}

\section{Proof of Theorem~\ref{MT}}
\label{s4}
In this section, we present the proof of our main result in this paper, Theorem~\ref{MT}. The proof requires the following Lemma:
\begin{lemma}
\label{sppat1}
The tuple $(\P,\L,\F,\pi_\P^\infty,\pi_\L)$ is a smooth projective plane. 
\end{lemma}

\begin{proof}[Proof of Theorem~\ref{MT}]
Let $(\P,\L,\F,\pi_{\P},\pi_{\L})$ be a two-dimensional smooth projective plane, then $\P$ is diffeomorphic to $\mathbb{RP}^2$ (Theorem~\ref{diffreal}). We endow $\P$ with a Riemannian metric of constant curvature $1$, making it isometric to $\mathbb{RP}^2$ with its usual metric. Define $\pi_\P^t$ as in (\ref{pipt}) using CSF. We prove in Proposition~\ref{submersion} that the limiting map $\pi_\P^\infty$ is smooth and $\pi_\P^t$ defines a smooth homotopy after reparametrization to the time interval $[0,1]$. Moreover, we show in Lemma~\ref{iso1} and Lemma~\ref{sppat1} that $(\P,\L,\F,\pi^t_\P,\pi_\L)$ is a smooth projective plane for every $t\in [0,1]$, after reparametrization. Note that as $t=1$, $(\P,\L,\F,\pi^1_\P,\pi_\L)$ is the real projective plane since for every $\ell\in\L$, $\pi_\P^1\pi_\L^{-1}(\ell)$ represents a geodesic in $\mathbb{RP}^2$. 
\end{proof}

The rest of this paper is devoted to the proof of Lemma~\ref{sppat1}. We will first prove that as $t\to\infty$, the number of transverse zeros of the solutions to the LCSF stays one. Since we begin with curves that defines a smooth projective plane, the number of transverse zeros of the solutions to the LCSF cannot increase. But this does not rule out the possibility that the solution to the LCSF vanish as $t\to \infty$. We will show in Lemma~\ref{con} that this cannot happen.

Recall that the linearized equation of (\ref{h1}) at $u$ is
\begin{equation}
\label{V1} v_t=(1+a(x,t))v_{xx}+b(x,t)v_x+(1+c(x,t))v
\end{equation} 
where \begin{equation}\label{res}||a(\cdot,t)||_{C^k},\;||b(\cdot,t)||_{C^k},\;||c(\cdot,t)||_{C^k}\text { are less than }\epsilon_k
e^{-t}
\end{equation}
and the initial condition satisfies
\begin{eqnarray}
    \label{P2}&&v(x+\pi)=-v(x) \text { for all  }  x\in\mathbb{R}\\
    \label{P3}&&v \text{ has exactly one transverse zero in every interval } [x,x+\pi)
\end{eqnarray}

\begin{lemma}
\label{con}If $v:\R \ra \R$ is
a smooth function satisfying (\ref{P2}) and (\ref{P3}) then the solution $v(x,t)$ to (\ref{V1}) satisfies
\begin{equation}\lim_{t\rightarrow\infty}||v(\cdot,t)||_{C^0}\neq
0.\end{equation}
\end{lemma}
The proof of Lemma~\ref{con} will be given in \S\ref{pf43}. 

\subsection{The linear growth property of Fourier coefficients}
\label{fc}
Define the $n^{th}$ Fourier coefficients of any $2\pi$-periodic, smooth, real valued function $f$ as
\begin{equation}
\label{fourierco0}
c_n=\frac{1}{2\pi}\int_0^{2\pi} f (x)\;e^{-inx}dx.
\end{equation}

\begin{lemma}
\label{LG} Let $f:\R \ra \R$ be a smooth function satisfying
conditions (\ref{P2}), (\ref{P3}). Then $|c_n|\leq \frac{\sqrt{2}}{2}n|c_1|$ and $|c_{-n}|\leq \frac{\sqrt{2}}{2}n|c_1|$ for all $n\in \N$.
\end{lemma}
\begin{proof}
We can assume that zeros of $f$ are $m\pi$, $m\in\Z$ (We could
translate $f$ in $x$ to make it true). Hence $f$ is either
nonnegative or non-positive in $[0,\pi]$. Let $a_n=\frac{1}{\pi} \int_0^{2\pi} f (x)\cos nx \;dx$ and $b_n=-\frac{1}{\pi}
\int_0^{2\pi} f(x) \sin nx \;dx$, then $c_n=\frac{a_n+ib_n}{2}$ and $c_{-n}=\frac{a_n-ib_n}{2}$.
\begin{eqnarray}
\label{a}
&&| a_{n+2}-a_{n}|\\
&=&|\frac{1}{\pi}\int_0^{2\pi} f(x)\cos((n+2)x)dx-\frac{1}{\pi}\int_0^{2\pi} f(x)\cos(n x)dx|\nonumber\\
&=&2|\frac{1}{\pi}\int_0^{2\pi} f(x)\sin x\sin((n+1)x)dx|\nonumber\\
&\leq&2\frac{1}{\pi}\int_0^{2\pi}| f(x)\sin x| dx\nonumber
\end{eqnarray}

By the condition (\ref{P3}) and the assumption that $f$ is nonnegative (or non-positive) in $[0,\pi]$, we have that $f(x)\sin x$ is always nonnegative (or non-positive) in $[0,2\pi]$. Hence 
\begin{eqnarray}
&&| a_{n+2}-a_{n}|\\
&\leq&2|\frac{1}{\pi}\int_0^{2\pi} f(x)\sin x dx|\nonumber\\
&=&2|b_1|\nonumber
\end{eqnarray}

A similar inequality can be derived for $b_n$.
\begin{eqnarray}
\label{b}
&&| b_{n+2}-b_{n}|\\
&=&|-\frac{1}{\pi}\int_0^{2\pi} f(x)\sin\left((n+2)x\right)dx+\frac{1}{\pi}\int_0^{2\pi} f(x)\sin(n x)dx|\nonumber\\
&=&2|\frac{1}{\pi}\int_0^{2\pi} f(x)\sin x\cos((n+1)x)dx|\nonumber\\
&\leq&2\frac{1}{\pi}\int_0^{2\pi} |f(x)\sin x |dx\nonumber\\
&=&2|\frac{1}{\pi}\int_0^{2\pi} f(x)\sin x dx|\nonumber\\
&=&2|b_1|\nonumber
\end{eqnarray}
By induction, (\ref{a}) and (\ref{b}) imply
\begin{eqnarray*} 
|a_n|&\leq& (n-1)|b_1|+|a_1|\;\;\text{ if
$n$ is odd and $n\geq 3$} \\
|a_n|&\leq& n|b_1|\;\;\text{ if $n$ is even}\\
|b_n|&\leq& n|b_1|
\end{eqnarray*}
Therefore,
\begin{eqnarray*}
|c_n|&=\frac{|a_n+ib_n|}{2}\leq \frac{\sqrt{2}}{2}n|c_1|\\
|c_{-n}|&=\frac{|a_n-ib_n|}{2}\leq \frac{\sqrt{2}}{2}n|c_1|
\end{eqnarray*}
for all $n\in \N$.
\end{proof}

\subsection{Infinite dimensional system of ODEs}
\label{ode}
We convert the PDE (\ref{V1}) into a linear system of ODEs by rewriting functions in \eqref{V1} in terms of their Fourier series and taking the inner product with $e^{inx}$ on both sides of the equation. Because of the orthonormality of $e^{inx}$, one can derive the following infinite
dimensional coupled system of ODEs. 
\begin{eqnarray}
\label{ODE1} \frac{d}{dt}v_n(t)&=&\Big((-n^2+1)+(-n^2)a_{0}+ i n\; b_{0}+
c_{0}\Big)v_n(t)\nonumber\\&&+\sum_{j\neq n}\Big( (-j^2)a_{n-j}+ ij\; b_{n-j}+
c_{n-j}\Big)v_j(t).
\end{eqnarray}

Let $a_\ell^{(k)}(t)$ be the $\ell^{th}$ Fourier coefficient for $a^{(k)}$, that is
\begin{equation}
a_\ell^{(k)}(t)=\frac{1}{2\pi}\int_0^{2\pi} a^{(k)}(x,t) e^{-i\ell x}\;dx.
\end{equation}
Applying integration by parts $k$ times, one can derive
\begin{equation}
\label{f0}
a_\ell^{(k)}(t)=(i\ell)^{k} a_\ell(t).
\end{equation}
On the other hand, from (\ref{res}) the condition that the $C^k$ norms of $a$, $b$ and $c$ are less than $\epsilon_k e^{-t}$, we know that
\begin{equation}
\label{f2}
|a_\ell^{(k)}(t)|=|\frac{1}{2\pi}\int_0^{2\pi} a^{(k)}(x,t) e^{-i\ell x} dx|\leq \frac{||a^{(k)}(\cdot,t)||_{C^0}}{2\pi}\int_0^{2\pi}|e^{-i \ell x}|dx\leq \epsilon_ke^{-t}.
\end{equation}
Hence (\ref{f0}) and (\ref{f2}) imply
\begin{equation}
\label{f3}
|(i\ell)^{k} a_\ell(t)|\leq  \epsilon_ke^{-t}.
\end{equation}
Note that by the same argument, we also have (\ref{f3}) for $b$ and $c$. Choose $k=6$, then

\begin{equation}
\label{AN}|a_\ell(t)|\leq \frac{\epsilon_6 e^{-t}}{\ell^6}, \quad |b_\ell(t)|\leq\frac{\epsilon_6 e^{-t}}{\ell^6},\quad |c_\ell(t)|\leq\frac{\epsilon_6 e^{-t}}{\ell^6},
\end{equation}
for all $\ell\neq 0$.

\subsection{Proof of Lemma~\ref{con}}
\label{pf43}
\begin{proof}
For any $t\geq 0$, since \begin{equation}\displaystyle\sup_{x}|v(x,t)|\geq \frac{1}{\sqrt{2\pi}}||v(x,t)||_{L^2}=(\sum_n |v_n(t)|^2)^{1/2}\geq |v_1(t)|,\end{equation} it is enough to show that $\lim_{t\to\infty}|v_1(t)|$ has a nonzero lower bound.

With $n=1$ in (\ref{ODE1}), we have
\begin{equation}\label{Case}\frac{d}{dt}v_1(t)=(-a_0 +i \; b_{0}+
c_{0})v_1(t)+\sum_{j\neq 1}( -j^2 a_{1-j}+ i\;j\; b_{1-j}+
c_{1-j})v_j(t).\end{equation} 
Therefore
\begin{eqnarray*}
\begin{split}
\frac{d |v_1|}{d t
}=&\frac{1}{2|v_1 |}\Big(\frac{d v_1}{dt}\bar{v_1}+\frac{d\bar{v_1}}{dt}v_1\Big)\\
=&\frac{1}{2|v_1 |}\Big(\Big((-a_0 +i \; b_{0}+ c_{0}) v_1+\sum_{j\neq
1} (-j^2 a_{1-j}+ i\; j\; b_{1-j}+
c_{1-j})v_j\Big)\bar{v_1}\\
&+\Big((-\bar{a}_0-i \; \bar{b}_{0}+\bar{c}_{0})\bar{v}_1+\sum_{j\neq 1}(
-j^2 \bar{a}_{1-j}- i\; j\;\bar{b}_{1-j}+ \bar{c}_{1-j})\bar{v}_j\Big)v_1\Big)\\
& (\text{By (\ref{Case})})\\
\geq&-(|a_0|+|b_0|+|c_0|)|v_1|-\frac{1}{|v_1 |}|v_1|\sum_{j\neq 1}(\;j^2\;|a_{1-j}|+|j||b_{1-j}|+|c_{1-j}|\;)|v_j|\\
&(\text{We use the fact that $z+\bar{z}\leq 2|z|$})\\
 \geq&\Big(\;-(|a_0|+|b_0|+|c_0|)-\frac{\sqrt{2}}{2}\sum_{j\neq 1}(\;|j|^3|a_{1-j}|+|j|^2|b_{1-j}|+|j||c_{1-j}|\;\Big)|v_1|\\
&(\text{By lemma \ref{fourierco0}},\; |v_j|\leq \frac{\sqrt{2}}{2}|j| |v_1|)\\
\geq&-\epsilon_6e^{-t}\Big(3+\frac{\sqrt{2}}{2}\sum_{j\neq 1}\frac{|j|^3+|j|^2+|j|}{(1-j)^6}\Big)|v_1|.\\
&(\text{By (\ref{AN})})
\end{split}
\end{eqnarray*}
Let $C=\epsilon_6\Big(3+\frac{\sqrt{2}}{2}\sum_{j\neq 1}\frac{|j|^3+|j|^2+|j|}{(1-j)^6}\Big)>0$. Then \begin{eqnarray}
\label{F} \frac{d |v_1|}{d t }\geq-Ce^{-t}|v_1|.
\end{eqnarray}
Note that at any $t\in [0,\infty)$, $|v_1|$ cannot be zero. Because if it is zero, then by Lemma~\ref{fourierco0}, we have $v=0$ at that time and this contradicts with Theorem~\ref{interpara}. 

Integrating (\ref{F}) with respect to $t$, we get that for all $t\geq 0$, 
\begin{equation}\label{nv}|v_1(t)|\geq e^{C(e^{-t}-
1)}|v_1(0)|\geq e^{-C
}|v_1(0)|
>0.
\end{equation} 
Hence $$\lim_{t\to\infty}|v_1(t)|\geq e^{-C
}|v_1(0)|.$$
\end{proof}
\subsection{Proof of Lemma~\ref{sppat1}}
\label{spp12}
\begin{proof}
We first verify that the tuple $(\P,\L,\F,\pi_\P^\infty,\pi_\L)$ satisfies SPP1 in Definition~\ref{SP1}. Pick a point $p$ in $\P$. Then for every $t\in[0,\infty]$ there is a line pencil $L_p^t$, consisting of the
lines passing through $p$; this is a subset of the manifold of lines $\L$.  By the
submersion property of $\pi_\P$, one can use the implicit function theorem to
say that the line pencil is a smooth 1-dimensional submanifold of $\L$, which
varies smoothly with $t$. Then there is a smooth map from $L_p^t$  to the
projectivized tangent space $P(T_p\P)$ of $\P$ at $p$, which takes a line $\ell$
in $L_p^t$ to its direction at $p$ (which is an element of $P(T_p\P))$. Since the
solutions to the LCSF have transverse zeroes, this map
$$
L_p^t \to P(T_p\P)
$$
is an immersion, and hence a covering map. As it varies continuously with
$t$, and is injective for $t < \infty$, it is injective when $t =\infty$. In particular,
distinct lines $\ell_1$, $\ell_2$ give rise to distinct point rows, which intersect
transversely. The number of intersection points of the point rows for
$\ell_1$ and $\ell_2$ varies continuously with $t$, and must therefore be $1$.

For every $\ell\in\L$, the restriction of $\pi^\infty_\P$ to $\hat{\ell}\subset\F$ is a smooth embedding. In addition, Lemma~\ref{con} implies that the number of transverse zero of the LCSF stays one. By Proposition~\ref{subvec}, $\pi_\P^\infty$ is a submersion.

Any two line pencils intersect at exactly one point $\ell$ in $\L$ because if they intersect at more than one point, then there are point rows intersecting at more than one point. By Lemma~\ref{con}, for any $0\neq\xi\in T_{\ell}\L $, the normal component of the corresponding variation of $\bar\ell_{\infty}$ vanishes at exactly one point. Hence, the transverse intersection of line pencils SPP2 follows from the same argument as in the proof of Lemma~\ref{iso1}.

By Definition~\ref{SP1} we can conclude that the tuple $(\P,\L,\F,\pi^\infty_{\P}.\pi_\L)$ is a smooth projective plane.
\end{proof}

\appendix
\section{}
\label{knestproof}

\begin{lemma}
\label{analysis1} Suppose $\{u_t : \R^m\times \R^n\to \R^\ell\}$, $t\in[0,\infty)$ is a family of smooth functions. Assume that for every $k$, the $k$th derivatives $\partial_x^k u_t $, $\partial_y^k u_t$ converge uniformly as $t\to\infty$, where $\partial_x^k$, $\partial_y^k$ refer to the derivatives in the $\R^m$ and $\R^n$ directions, respectively.
Then for every $k$, the family $\{u_t\}$converges in the $C^k$-topology as $t\to\infty$. 
\end{lemma}
\begin{proof}
It suffices to prove the following claim: if $u$ is a continuous function on the n-cube $[0,2\pi]^n$, such that the partial derivatives in the coordinate directions 
$\frac{\partial^k u}{\partial x_i^k}$ exist and are continuous for all $i$ and $k$, then $u$ is smooth on the subcube $ [\frac{\pi}{2},\frac{3\pi}{2}]^n$. Furthermore, for every $k$, if $N>n/2+k$, then the $C^k$ norm of $u$ on $[\frac{\pi}{2},\frac{3\pi}{2}]^n$, $|| u ||_{C^k( [\frac{\pi}{2},\frac{3\pi}{2}]^n)}$, is bounded by the $C^0$ norm of the partial derivatives $\frac{\partial^{m} u }{\partial x_i^{m}}$ in $ [0.2\pi]^n$, for all $i$ and $m\leq N$.

To prove this claim one can take a smooth cutoff function $\phi:[0,2\pi]^n\to\R$ which is $1$ on a neighborhood of $[\frac{\pi}{2},\frac{3\pi}{2}]^n$, and has support in $[\frac{\pi}{4},\frac{7\pi}{4}]^n$, and verify the claim for $\phi u$ instead. 

The function $\phi u$ defines a $L^2$ function on $\mathbb{T}^n=\R^n/(2\pi \Z^n)$, the $n$-torus. Let $\{E_J(\mathbf{x})=e^{ i J\cdot\mathbf{x}}\;|\;J=(j_1,...,j_n)\in\Z^n,\;\mathbf{x}\in \mathbb{T}^n\}$ be an orthonormal basis for $L^2(\mathbb{T}^n)$. We define the Fourier coefficient of $\phi u$ by $c_{j_1...j_n}=\widehat{\phi u}(J)=\int_{\mathbb{T}^n} (\phi u)(\mathbf{x}) E_{-J} (\mathbf{x}) d\mathbf{x}$. 

Since the distribution derivatives of $\phi u$ of all orders in the coordinate directions are in $L^2$, its Fourier coefficients ${c_{j_1...j_n}}_{\{j_\ell\in \Z\}}$ are square summable with the weight  $|j_\ell|^k$, for any $k$. This implies that the Fourier coefficients decay faster than any polynomial. By the Sobolev embedding theorem, we have $\phi u\in C^{\infty}(\mathbb{T}^n)$. Furthermore, for every $k$, if $N>n/2+k$, then $||\phi u ||^2_{C^k}\leq \sum_{\{\alpha | |\alpha|\leq N\}}||D^{\alpha} \phi u||^2_{L^2}$ where $\alpha=(\alpha_1,...,\alpha_n)$ is an $n$-dimensional multi-index of non-negative integers, $|\alpha|=\alpha_1+...+\alpha_n$ and $D^{\alpha}=\frac{\partial^{|\alpha|}}{\partial_{x_1}^{\alpha_1}...\partial^{\alpha_n}_{x_n}}$. To prove the second part of the claim, it is enough to show that $\sum_{\{\alpha| |\alpha|\leq N\}}||D^{\alpha} \phi u||^2_{L^2}\leq C\sum_{m=0}^N ||\partial_{x_\ell}^m \phi u||_{L^2}^2$, for some constant $C=C(k,n)$. 
\begin{eqnarray*}
\label{fcoe}
\begin{split}
\sum_{\{\alpha| |\alpha|\leq N\}}||D^{\alpha} \phi u||^2_{L^2}=&\sum_{\{\alpha| |\alpha|\leq N\}}||\widehat{D^{\alpha} \phi u}||^2_{L^2}\\
=&\sum_{\{|\alpha||\alpha|\leq N\}}\sum_{j_\ell\in\Z} (j_1)^{2\alpha_1}...(j_n)^{2\alpha_n}|c_{j_1...j_n}|^2\\ 
\leq&\sum_{\{|\alpha| |\alpha|\leq N\}} \sum_{j_\ell\in\Z} \Big((j_1)^{2|\alpha|}+...+(j_n)^{2|\alpha|}\Big)|c_{j_1...j_n}|^2\\
=&\sum_{m=0}^N {m+n-1 \choose m} \sum_{j_\ell\in\Z} \Big((j_1)^{2m}+...+(j_n)^{2m}\Big)|c_{j_1...j_n}|^2\\
\leq&C\sum_{m=0}^N\sum_{i=1}^n||\widehat{\partial_{x_i}^{m} \phi u }||_{L^2}^2\\
\leq& C\sum_{m=0}^N\sum_{i=1}^n||\partial_{x_i}^{m} \phi u ||^2_{C^0}.
\end{split}
\end{eqnarray*}
\end{proof}

\bibliographystyle{amsplain}
\bibliography{RSreference}

\end{document}